\documentclass[10pt,reqno]{amsart}
\usepackage[utf8]{inputenc}
\usepackage{microtype}
\allowdisplaybreaks%
\usepackage{graphicx} 
\usepackage[a4paper,margin=2.2cm]{geometry}
\usepackage{fourier,amssymb}
\usepackage{enumitem}
\usepackage{tikz}
\usetikzlibrary{arrows.meta}
\usepackage[authoryear]{natbib}
\usepackage{xcolor}
\usepackage[hidelinks]{hyperref}

\usepackage{stmaryrd}
\renewcommand{\Re}{\operatorname{Re}}
\renewcommand{\Im}{\operatorname{Im}}
\newcommand{\R}{{\mathbb{R}}}
\newcommand{\N}{{\mathbb{N}}}

\newcommand{\C}{{\mathbb{C}}}
\newcommand{\A}{{\tilde{A}}}
\newcommand{\B}{{\tilde{B}}}
\newcommand{\II}{I^{\geq \eta_0}} 
\newcommand{\dd}{\mathrm{d}}
\newcommand{\ii}{\mathrm{i}}

\newcommand{\ie}{i.e., }
\newcommand{\eg}{e.g., }

\newcommand{\wt}{\widetilde}

\newcommand{\<}{\langle}
\renewcommand{\>}{\rangle}
\newcommand{\G}{\mathrm{Gin}}
\newcommand{\AG}{A^{\mathrm{Gin}}}
\newcommand{\BG}{B^{\mathrm{Gin}}}
\newcommand{\MG}{M^{\mathrm{Gin}}}
\title{On the spectral radius of the ratio of Girko matrices}
\date{Autumn 2025; Revised Winter 2026}
\author{Djalil Chafaï}
\address{DMA, École normale supérieure, CEREMADE, Université Paris-Dauphine, PSL, CNRS, France}
\email{djalil.chafai@ens.psl.eu}
\author{David García-Zelada} 
\address{LPSM, Sorbonne Université, Paris, France}
\email{david.garcia-zelada@sorbonne-universite.fr}
\author{Yuan Yuan Xu}
\address{Academy of Mathematics and Systems Science,
Chinese Academy of Sciences, Beijing, China}
\email{yyxu2023@amss.ac.cn}
\numberwithin{equation}{section}
\newtheorem{theorem}{Theorem}[section]%
\newtheorem{lemma}[theorem]{Lemma}%
\newtheorem{remark}[theorem]{Remark}%
\newtheorem{example}[theorem]{Example}%
\begin{document}
\begin{abstract} 
  Girko matrices have independent and identically distributed entries of mean zero and unit variance. In this note, we consider the random matrix model formed by the ratio of two independent Girko matrices. Such a random matrix is heavy tailed with dependent entries. Our main message is that divided by the square root of the dimension, the spectral radius of the ratio converges in distribution, when the dimension tends to infinity, to a universal heavy-tailed distribution. We provide 
  a mathematical proof of this high-dimensional phenomenon, under a fourth moment matching with a Gaussian case known as the complex Ginibre ensemble. In this Gaussian case, the model is known as the spherical ensemble, and its spectrum is a determinantal planar Coulomb gas. Its image by the inverse stereographic projection is a rotationally invariant gas on the two-sphere. 
  A crucial observation is the invariance in law of the model under inversion, related to its spherical symmetry, and that makes, in a sense, edge and bulk equivalent. Our approach involves Girko Hermitization, local law estimates for Wigner matrices, lower bound estimates on the smallest singular value, and convergence of kernels of determinantal point processes. The universality of the high-dimensional fluctuation of the spectral radius of the ratio of Girko matrices turns out to be remarkably more accessible mathematically than for a single Girko matrix!
\end{abstract}
\maketitle

{\footnotesize\tableofcontents}


\section{Introduction}

The asymptotic analysis in high dimension of the spectral radius of non-normal
random matrices is a delicate subject, which goes back at least to
\cite{zbMATH03864263,zbMATH03978046,zbMATH03940306,zbMATH02028588}. There has been recent progress in the case of
Girko random matrices, see for instance
\cite{MR4408512,MR4761213,cipolloni2024universalityextremaleigenvalueslarge,zbMATH07925435}
and references therein. Girko matrices have independent square-integrable
entries. This note addresses the case of the ratio of two independent Girko
matrices, a model of heavy-tailed non-normal random matrices.
For this model, we prove the universality of
the asymptotic behavior of the spectral radius in high dimension. This model,
in the special case of Ginibre matrices, known as the spherical ensemble,
becomes integrable and biunitary invariant, and the law of its spectrum is a
planar heavy-tailed Coulomb gas with a determinantal structure. The spherical
ensemble is a matrix analogue and extension of the ratio of independent real
Gaussian random variables, and can be then seen as a sort of matrix Cauchy law.

The \emph{eigenvalues} of an $n\times n$ complex matrix
$A=(A_{jk}:1\leq j,k\leq n)\in\mathcal{M}_n(\mathbb{C})$ are the roots in
$\mathbb{C}$ of its characteristic polynomial
$\det(A-z\mathrm{Id}_n)\in\mathbb{C}[z]$. They form a multiset of cardinal $n$
called the \emph{spectrum} of $A$, denoted $\mathrm{spec}(A)$. The
\emph{unnormalized empirical spectral distribution}, \emph{spectral radius}, and
\emph{reciprocal or inner spectral radius} of $A$ are\footnote{When
  $A$ is invertible, $\rho_{\min}(A)=\rho_{\max}(A^{-1})^{-1}$, a formula
  that remains valid in an extended sense when $A$ is singular.}
\begin{equation}\label{def_esd}
\mu_A=\sum_{\lambda\in\mathrm{spec}(A)}\delta_{\lambda},
\qquad
\rho_{\max}(A)=\max_{\lambda\in\mathrm{spec}(A)}|\lambda|,
\quad\text{and}\quad
\rho_{\min}(A)=\min_{\lambda\in\mathrm{spec}(A)}|\lambda|.
\end{equation}
The reason why we choose not to normalize $\mu_A$ by $1/n$ is the central role played by point processes in our work.

\subsection{Main results}

We say that an $n\times n$ random matrix $A=(A_{jk}: 1\leq j,k\leq n)$ is a \emph{Girko matrix} when its entries are independent and identically distributed (i.i.d.) in $\mathbb{C}$ with $\mathbb{E}(A_{11})=0$ and $\mathbb{E}(|A_{11}|^2)=1$. We say that $A$ is a \emph{complex Ginibre matrix} in the special case where $A_{11}\sim\mathcal{N}_{\mathbb{C}}(0,\frac{1}{2}\mathrm{Id}_2)$, which is equivalent to state that $A$ has density in $\mathbb{C}^{n\times n}$ proportional to $A\in\mathcal{M}_n(\mathbb{C})\mapsto\exp(-\mathrm{Tr}(AA^*))$. To distinguish, we denote such a Ginibre matrix by $\AG$, and we always write $\mathrm{Gin}$ as a superscript to indicate this special case.

\smallskip
We consider the following assumptions or conditions on an $n\times n$ Girko matrix $A$: 
\begin{enumerate}[label=(C\arabic*)]
\item\label{it:condensity} the law of $A_{11}$ has a bounded density function $\varphi$ with respect to the Lebesgue measure on $\mathbb{C}$, and
\begin{equation*}
\|\varphi\|_\infty\leq D_0\quad\text{ for some constant $D_0>0$.}
\end{equation*}
\item\label{it:condmom4} \emph{Fourth moment matching condition.}
There exists a small constant $c_0>0$ such that
\begin{equation*}
\mathcal{E}_1=\mathcal{E}_2=0, \quad \qquad \mathcal{E}_3 \leq n^{-\frac{1}{2}-c_0}, \quad\qquad \mathcal{E}_4 \leq  n^{-c_0},
\end{equation*}
where ($k$-th moment deviation of $A_{11}$ from standard complex Gaussian)
\begin{equation*}
\mathcal{E}_k=\max_{\substack{a,b\in\N\\a+b=k}} \left|\mathbb{E}[(\Re A_{11})^a(\Im A_{11})^b]-\mathbb{E}[(\Re \AG_{11})^a(\Im \AG_{11})^b]\right|.
\end{equation*}

\item\label{it:condmom} \emph{Finite moment condition.} 
For each $k\geq 1$, there exists constants $D_k>0$ (independent of $n$) such that
\begin{equation*}
\mathbb{E}[|A_{11}|^k] < D_k.
\end{equation*}

\end{enumerate}

 Condition \ref{it:condensity} implies that the $n\times n$ complex matrix $A$ is a.s.\ invertible for all $n$ since the algebraic hypersurface $\{\det=0\}$ has zero Lebesgue measure. Moreover, the uniform bound on the density allows a precise  lower bound on the smallest singular value, see \eg \cite{Davies}. More generally, we refer to \cite{MR4680362} for the invertibility of random matrices beyond the density case.

Condition \ref{it:condmom4} states that the moments of $A_{11}$ almost match those of $\AG_{11}$ up to the fourth order.

Condition \ref{it:condmom} is assumed for technical reasons to simplify a proof.

\begin{example}[Moment matching with the standard complex Gaussian distribution]
An example of a probability measure $\nu$ on $\mathbb{C}\equiv\mathbb{R}^2$ with moment matching up to order $k\geq1$ with
$\mu=\mathcal{N}(0,\frac{1}{2}\mathrm{Id}_2)$ is 
\begin{equation}
\dd\nu=(1+f)\dd\mu
\end{equation}
where $f\in L^2(\mu)$ is such that 
$f\neq0$, $f\geq-1$, and $f\perp\mathbb{R}_k[X,Y]$ where $\mathbb{R}_k[X,Y]$ is the set of bivariate polynomials with total degree less than or equal to $k$, the span of $X^aY^b$ with $a,b\in\mathbb{N}$, $a+b\leq k$. Indeed, for all $P\in\mathbb{R}_k[X,Y]$,
\begin{equation}
\int P\dd\nu-\int P\dd\mu=\langle P,f\rangle_{L^2(\mu)}=0.
\end{equation}
The special case $P=1$ states that $\nu$ is a probability measure. 
However, this construction has no reason to produce a product measure. Also, we may prefer to take alternatively 
\begin{equation}
\dd\nu=g\dd\mu
\quad\text{where}\quad
g(z)=g(x+\ii y)=(1+f_1(x))(1+f_2(y))
\end{equation}
where $f_1,f_2:\mathbb{R}\to\mathbb{R}$ are in $L^2(\mathcal{N}(0,\frac{1}{2}))$ and satisfy
$f_i\neq0$, $f_i\geq-1$, and $f_i\perp\mathbb{R}_k[X]$ where $\mathbb{R}_k[X]$ is the span of univariate polynomials with total degree less than or equal to $k$. A basic example is given by $f_1=f_2=c^{-1}H_m$ where $H_m$ is the Hermite polynomial, orthogonal with respect to $\mathcal{N}(0,\frac{1}{2})$, of even degree $m>k$, and $c=|\inf_{\mathbb{R}}H_m|>0$. In this case $\nu$ has a bounded density with respect to the Lebesgue measure on $\mathbb{C}\equiv\mathbb{R}^2$. The moment matching up to order $k$ between $\mu$ and $\nu$ comes from $\mu=\mathcal{N}(0,\frac{1}{2})^{\otimes 2}$.
In another direction, by using the Gauss--Hermite quadrature, it is possible to construct a finite discrete distribution that matches the moments of the Gaussian up to an arbitrary fixed order.
\end{example}

This note is concerned with the space of 
point configurations endowed with the 
weakest topology such that the maps
$\mu_{\mathcal X} \mapsto \int f\mathrm d \mu_{\mathcal X}$
are continuous
for every continuous
compactly supported function $f$.
As a consequence, the convergence in distribution
for a sequence of point processes
is always taken with respect to this topology.

\begin{theorem}[From ratio of Girko to infinite Ginibre]\label{th:spherical:girko:gininf}
 Let $M=AB^{-1}$ where $A$ and $B$ are independent $n\times n$ Girko matrices, not necessarily of the same law, both satisfying \ref{it:condensity}--\ref{it:condmom}. Let $\mathrm{Gin}_\infty$ be the infinite Ginibre determinantal point process on $\mathbb{C}$, with kernel 
 $K(z,w)=\sqrt{\gamma(z)\gamma(w)}\mathrm{e}^{z\overline{w}}$ where $\gamma(z)=\frac{1}{\pi}\mathrm{e}^{-|z|^2}$. 
 Then, for all fixed $\lambda_0 \in \mathbb C$, 
    \[
    \mu_{\sqrt{n}(1+|\lambda_0|^2)^{-1}(M-\lambda_0\mathrm{Id})}=\sum_{\lambda\in\mathrm{spec}(M)}\delta_{\sqrt{n}(1+|\lambda_0|^2)^{-1}(\lambda - \lambda_0)}
    \xrightarrow[n\to\infty]{\dd}
    \mathrm{Gin}_\infty.
    \]
\end{theorem}


\begin{theorem}[Spectral radius]\label{th:spherical:girko:rho}
 Let $M$ be as in Theorem \ref{th:spherical:girko:gininf}, and ${(\gamma_k)}_{k\geq1}$ independent with $\gamma_k\sim\mathrm{Gamma}(k,1)$. Then
 \begin{equation}\label{eq:spherical:girko:rho}
    \frac{\rho_{\max}(M)}{\sqrt{n}} 
    \xrightarrow[n\to\infty]{\dd}
    \mathrm{Law}\Bigr(\frac{1}{\sqrt{\min_{k\geq1}\gamma_k}}\Bigr)
    \quad\text{and}\quad
    \sqrt{n}\rho_{\min}(M)
    \xrightarrow[n\to\infty]{\dd}
    \mathrm{Law}\bigr(\sqrt{\min_{k\geq1}\gamma_k}\bigr).
 \end{equation}
\end{theorem}

\begin{remark}[Tail]
The laws of $R_0=\sqrt{\min_{k\geq1}\gamma_k}$ and $R_\infty=\displaystyle\frac{1}{R_0}$ have respective cdf given for all $x>0$ by 
\begin{equation}\label{eq:cdf}
 \mathbb{P}(R_0\leq x)
 =1-\prod_{k=1}^\infty (1-F_k)(x^2)
 \quad\text{and}\quad
 \mathbb{P}(R_\infty\leq x)
 =\prod_{k=1}^\infty (1-F_k)(x^{-2})
\end{equation}
where 
$1-F_k(x)=\mathbb{P}(\gamma_k\geq x)=\mathbb{P}(\pi_x<k)=\mathrm{e}^{-x}\sum_{j=0}^{k-1}\frac{x^{j}}{j!}$
with $\pi_x\sim\mathrm{Poisson}(x)$.
In particular $R_\infty$ is heavy-tailed as 
\begin{equation}
\mathbb{P}(R_\infty>x)\underset{x\to+\infty}{\sim}\frac{1}{x^2}.    
\end{equation}
We are beyond classical extreme value theory: ${(\xi_k)}_{k\geq1}$ are not i.i.d., $R_0$ is not Weibull, and $R_\infty$ is not Fréchet.
\end{remark}

\begin{remark}[Point process of farthest particles]
Actually, the proof of Theorem 
\ref{th:spherical:girko:rho} gives much more: we not
only obtain the limit of the farthest particle,
but also of the second farthest particle, 
the third and so on.
We may formulate this as saying that for every
continuous function $f:\mathbb C \to \mathbb R$
whose support does not contain the origin,
\[\sum_{\lambda \in \mathrm{spec}(M)}
f(\lambda/\sqrt n) \xrightarrow[n \to \infty]
{\mathrm d} \sum_{w \in \mathrm{Gin}_\infty} f(1/w).\]
Since we do not ask $f$ to be compactly supported,
we are able to obtain the information
of the farthest particles.
\end{remark}

Theorem \ref{th:spherical:girko:rho} establishes, for the first time and to some extent, the universality of the high-dimensional behavior of the spectral radius of the spherical ensemble, studied by \cite{MR3615091} when $A$ and $B$ are complex Ginibre matrices, see Section \ref{se:spherical}. Theorem \ref{th:spherical:girko:rho} is an analogue for the spherical ensemble of the universality of the Gumbel fluctuation for the complex Ginibre ensemble obtained by \cite{cipolloni2024universalityextremaleigenvalueslarge}. 

The proofs of Theorem \ref{th:spherical:girko:gininf} and Theorem \ref{th:spherical:girko:rho} are given in Section \ref{se:proof:th:spherical:girko:gininf+rho}. The strategy is as follows. By using the invariance of the model by inversion, we transform the behavior of the dilated spectral radius into the local behavior near the origin, more precisely a local gap probability at the origin. We then use the universality of the local behavior near the origin, stated as a replacement principle in Theorem \ref{th:replacement:simple} below. This reduces the problem to the local gap probability of the spherical ensemble, which in turn reduces to the infinite Ginibre ensemble by using kernel convergence of determinantal point processes. Finally, the local gap probability for the infinite Ginibre ensemble follows from the Kostlan observation of the moduli of the particles of this determinantal point process.

The spherical or Forrester--Krishnapur ensemble or model of dimension $n$ is the random matrix defined by
\begin{equation}\label{eq:spherical}
    \MG=\AG (\BG)^{-1}
\end{equation}
where $\AG$ and $\BG$ are independent $n\times n$ complex Ginibre matrices with i.i.d.\ $\mathcal{N}_{\mathbb{C}}(0,\frac{1}{2}\mathrm{Id}_2)$ entries.

\begin{theorem}[Replacement principle for local eigenvalue statistics near a point]\label{th:replacement:simple}  
Let $M=AB^{-1}$ be as in Theorem \ref{th:spherical:girko:gininf} and $\MG$ as in \eqref{eq:spherical}. Let $\mu_M$ and $\mu_{\MG}$ be the unnormalized empirical spectral measures as in \eqref{def_esd}. 
Then the following replacement principle on scale $\sqrt{n}$ occurs: for all $z_0 \in \mathbb C$ and all compactly supported $\mathcal{C}^\infty$ functions $f:\mathbb{C}\to\mathbb{R}$ and $F:\mathbb{R}\to\mathbb{R}$, 
    \[
    \lim_{n \rightarrow \infty}\Bigr(\mathbb{E}\Bigr[F\Bigr(  \int f(\sqrt n (z-z_0))\dd\mu_M(z)\Bigr)\Bigr]
    -\mathbb{E}\Bigr[F\Bigr( \int f(\sqrt n (z-z_0))\dd\mu_{\MG}(z)\Bigr)\Bigr]\Bigr) =0.
    \]
\end{theorem}

Theorem \ref{th:replacement:simple} is a simplified version of a more technical result stated in Theorem \ref{th:replacement} and proved in Section \ref{se:proof:th:replacement:simple}. It can be seen as a fourth moment theorem for the spherical ensemble. 
The proof involves Ornstein--Uhlenbeck interpolation between Gaussian and non-Gaussian, a cumulant expansion exploiting the fourth moment matching \ref{it:condmom4}, a local law for Wigner matrices and its consequence in terms of rigidity, and a precise lower bound on the smallest singular value.
The continuous interpolation that we use, based on the Ornstein--Uhlenbeck diffusion, differs from the original Lindeberg-type entry by entry replacement used by \cite{MR3306005} in their original fourth moment theorem. 

\begin{remark}[Real case]\label{rem:real}
 Though lacking analogous versions of Theorems \ref{th:spherical:girko:gininf}-\ref{th:spherical:girko:rho} for the real spherical ensemble, the replacement principle stated in Theorem \ref{th:replacement:simple} applies to the 
real case where $M=AB^{-1}$ with $A,B$ being real Girko matrices, and $A^{\mathrm{Gin}}$ and
$B^{\mathrm{Gin}}$ being real Ginibre matrices. The proof is essentially the same, with slight modifications.
  Indeed, the eigenvalues of the real spherical ensemble can be studied using Pfaffians instead
 of determinants as explained in
 \cite{zbMATH06125015}.

\end{remark}

\subsection{Comments and open problems}

\subsubsection{Model.} 
We could use the alternative definitions $M=A^{-1}B$ instead of $M=AB^{-1}$. Indeed $XY$ and $YX$ have same spectrum thanks to the Sylvester identity $\det(XY-z\mathrm{Id})=\det(YX-z\mathrm{Id})$. 

\subsubsection{First order universality.}
The universality of the limiting spectral distribution of the model was proved in \cite{MR2772389}, using logarithmic potential, Girko Hermitization, and Tao and Vu replacement principle. The limit can be interpreted by using the notions of freeness and Brown measure from Free Probability Theory.

\subsubsection{Inner radius of Girko matrices.}
The analogue of Theorem \ref{th:spherical:girko:gininf} for a single Girko matrix $A$ follows from \cite[Th.~2.1~and~2.2]{maltsev-osman}, under a second moment matching assumption. This leads to the universality of the fluctuation of the inner spectral radius, which meets the gap at the origin of $\mathrm{Gin}_\infty$, namely
\begin{equation}
    \sqrt{n}\rho_{\min}(A)
    \xrightarrow[n\to\infty]{\dd}
    \mathrm{Law}\Bigr(\sqrt{\min_{k\geq1}\xi_k}\Bigr)
    \quad\text{where}\quad
    \text{${(\xi_k)}_{k\geq1}$ are independent with $\xi_k\sim\mathrm{Gamma}(k,1)$.}
\end{equation}

\subsubsection{Moments assumptions.} The fourth moment matching in \ref{it:condmom4} helps to reduce the complexity of the proof of Theorem \ref{th:replacement}, essentially for  the bulk universality for Girko matrices, see for instance \cite{MR3306005}. Since bulk universality for Girko matrices with a second moment matching was proved recently in \cite{maltsev-osman}, we expect that Theorem \ref{th:spherical:girko:rho} also holds 
under the assumption of second moment matching only. 

\subsubsection{Characteristic polynomials}
It is natural to ask about the asymptotic analysis of the characteristic polynomial of $AB^{-1}$, in the spirit of
\cite{MR4408512}, with a random analytic object as a limit. Taking the modulus and the logarithm would then recover the CLT for the log-potential. This seems to be open even for the spherical model.

\subsubsection{Ratio of other models} We could ask about the behavior of the spectral radius of $AB^{-1}$ if $A$ and $B$ are independent real or quaternion Girko matrices, and, in another direction, real/complex/quaternion Wigner matrices. As mentioned in Remark \ref{rem:real}, the replacement principle is expected to be true with slight modifications, though the statements for the ratios of Ginibre/Gaussian ensembles are missing.

\subsection{Conventions and notation}

The convergence in distribution of a sequence of random variables is the weak convergence of their laws, with respect to continuous and bounded test functions, towards a probability distribution. We denote it by a roman letter $\dd$ over ``$\to$''. Similarly, we use $\dd$ over ``$=$'' for the equality in distribution. 

For integers $k,l\in\N$, with $k<l$, we write 
$\llbracket k,l \rrbracket= \{k,\dots, l\}$ and 
$\llbracket k \rrbracket=
\llbracket 1,k \rrbracket$.

We set $\dd^2 z= \frac{1}{2}\ii(\dd z\wedge \dd \overline{z})$ for the two dimensional area form on $\C$, matching the Lebesgue measure $\dd x\dd y$ on $\mathbb{R}^2$.

The two-dimensional Laplace operator is $\Delta=\Delta_z=4\partial_z \partial_{\bar z}=\partial^2_{\Re z}+\partial^2_{\Im z}$ denotes the Laplace operator. 

For any vector $\mathbf{x} \in \C^{d}$ and matrix $A\in \C^{d \times d}$, $d\in \N$, we use $\lVert\mathbf{x}\rVert$ and $\lVert A\rVert$ indicate the usual Euclidean norm and the corresponding induced matrix norm (also known as operator norm). 
Throughout the paper $c,C>0$ denote small and large constants, respectively, which may change from line to line.  For positive quantities $f,g$ we write $f\lesssim g$ and $f\sim g$ if $f \le C g$ or $c g\le f\le Cg$, respectively, for some constants $c,C>0$ independent of $n$. Furthermore, for $n$--dependent positive
quantities $f_n,g_n$ we use the notation $f_n\ll g_n$ to denote that $\lim_{n\to\infty} (f_n/g_n)=0$.

We will often use the concept of ``with very high probability'' for an $n$-dependent event, meaning that for any fixed $D>0$ the probability of the event is bigger than $1-n^{-D}$ when $n\ge n_0(D)$. We recall the standard notion of \emph{stochastic domination}:  
given two families of non-negative random variables
\[
X=\left(X^{(n)}(u) \,:\, n\in\N, u\in U^{(n)}\right)\quad\text{and}\quad Y=\left(Y^{(n)}(u) \,:\, n\in\N, u\in U^{(n)}\right)
\] 
indexed by $n$ (and possibly some parameter $u$  in some parameter space $U^{(n)}$), 
we say that $X$ is {\it stochastically dominated} by $Y$, if for any small $\xi$ and large $D>0$ we have 
\begin{equation}
	\label{prec}
	\sup_{u\in U^{(n)}} \mathbb{P}\left[X^{(n)}(u)>n^\xi  Y^{(n)}(u)\right]\le n^{-D}
\end{equation}
for large enough $n\geq n_0(\xi,D)$. In this case we use the notation $X\prec Y$ or $X= O_\prec(Y)$
and we say that $X(u)\prec Y(u)$ {\it holds uniformly} in $u\in U^{(n)}$. Properties of stochastic domination can be found in \cite[Proposition 6.5]{zbMATH06780221}. We often use the notation $\prec$ also for deterministic quantities, then~\eqref{prec} holds with probability one. 

\subsection*{Acknowledgments}

{\small The authors were supported by the Swedish Research Council under grant no.\ 2021-06594 while in residence at Institut Mittag-Leffler in Djursholm, Sweden during the fall semester of 2024, Program "Random Matrices and Scaling Limits" organized by 
Maurice Duits, Kurt Johansson, Gaultier Lambert, and Kevin Schnelli, KTH Royal Institute of Technology. Y.X. is partially supported by Grant No. 2024YFA1013503 from National Key R\&D Program of China.}

\begin{figure}[htbp]
 \centering
 \includegraphics[width=.7\textwidth]{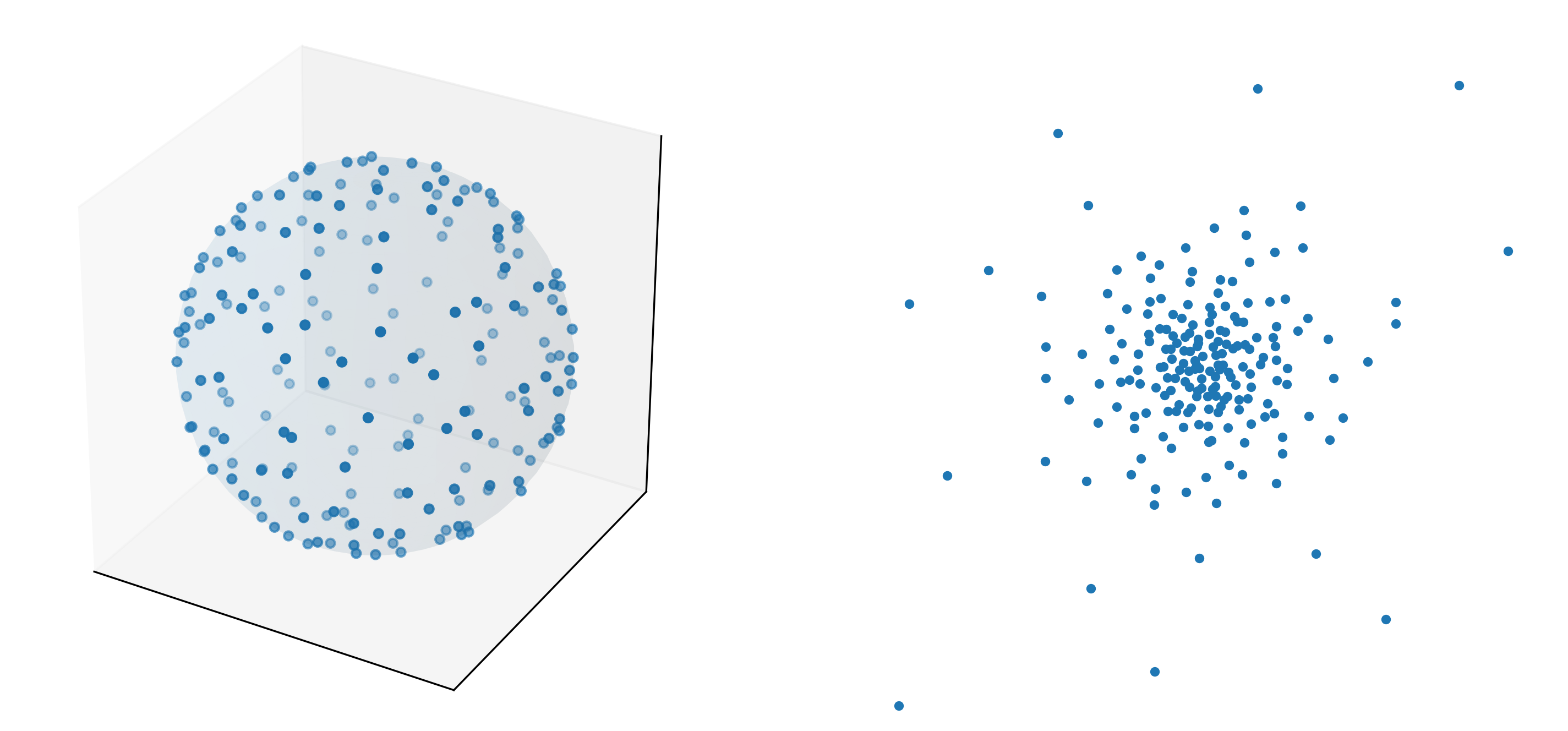}
 \caption{\label{fi:spherical}Sample of the spherical ensemble and its planar stereographic projection.}
\end{figure}

\begin{figure}[htbp]
  \centering
  \begin{tikzpicture}[scale=1.8]
    \draw[thick] (0,0) circle(1);
    
    \draw[->] (-1.2,0) -- (4,0) node[right] {$z$};
    \draw[->] (0,-0.2) -- (0,1.4) node[above] {$x_3$};
    
    \coordinate (N) at (0,1);               
    \coordinate (O) at (0,0);               
    \coordinate (P) at ({sqrt(2)/2}, {sqrt(2)/2});      
	\coordinate (Pproj) at ({1/(sqrt(2)-1)}, 0);            
    
    \draw[dashed] (N) -- (Pproj);          
    
    \fill (O) circle(0.04) node[below left] {$0$};
    \fill (N) circle(0.04) node[above right] {$N$};
    \fill[blue] (P) circle(0.04) node[above right] {$P$};
    \fill[blue] (Pproj) circle(0.04) node[below] {$T(P)$};
  \end{tikzpicture}
  \caption{\label{fi:stereo}The point $T(P)\in\mathbb{C}$ is the image of
    $P\in\mathbb{S}^2\setminus\{N\}$ by the stereographic projection or
    transform $T$ with respect of the north pole $N$, see \eqref{eq:T}. See
    also \cite{zbMATH07645442}. The image of $N$ is $\infty$. For gases, the
    gap around $N$ on $\mathbb{S}^2$ is in correspondence with the gap around
    $\infty$ on $\mathbb{C}\cup\{\infty\}$.}
\end{figure}
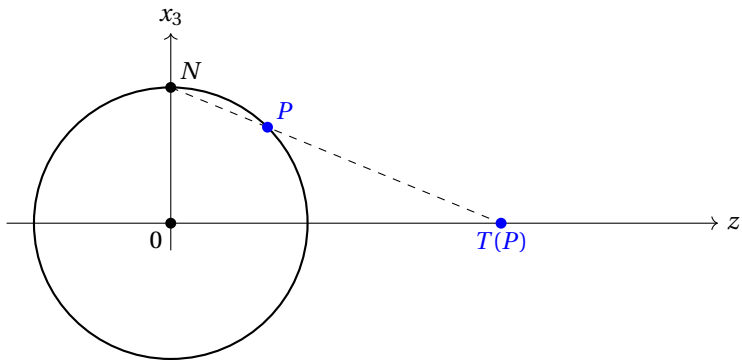

\begin{figure}[htbp]
    \centering
    \includegraphics[width=0.6\linewidth]{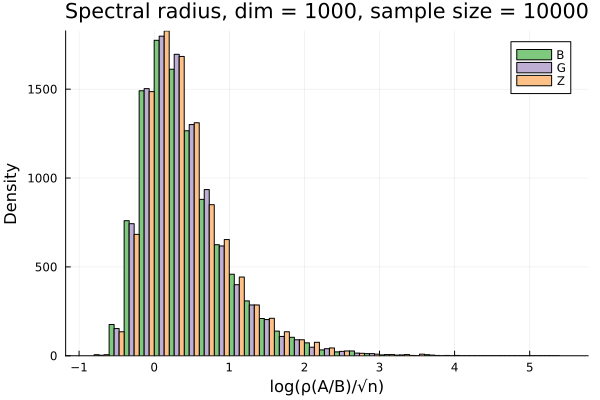}
    \includegraphics[width=0.6\linewidth]{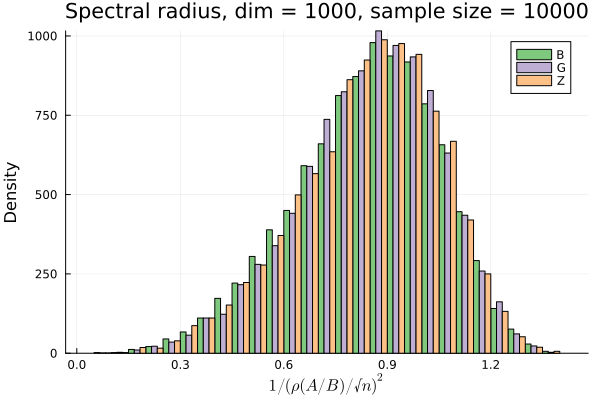}    
    \caption{Histograms illustrating the universality of the high dimensional behavior, for three choices of entries law : symmetric complex Bernoulli or Rademacher, isotropic complex Gaussian, and symmetric complex discrete heavy tailed (same tail as Zipf or Pareto, $\alpha=4$). For the Bernoulli and Zipf cases, the matrices are also conditioned to be nonsingular by using rejection sampling.}
    \label{fig:placeholder}
\end{figure}

\section{The spherical ensemble}\label{se:spherical} 

This section gathers useful properties of the spherical ensemble, in potential theory, geometry, and probability.

\subsection{Coulomb gas and determinantal structure}

The spherical or Forrester--Krishnapur random matrix ensemble \eqref{eq:spherical} is considered for instance in \cite{MR2552864,MR2641363}. The law of $\MG$ inherits the biunitary invariance : invariance by multiplication from the left and from the right by deterministic unitary matrices. The spectrum of $\MG$ forms a Coulomb gas on $\mathbb{C}$ with density\footnote{By conjugacy with an independent Haar unitary matrix, this can also be seen as the spectrum of a random normal matrix model, which should not be confused with the random non-normal matrix model obtained by taking the ratio of two independent complex Ginibre matrices.} 
$\varphi_n:\mathbb{C}^n\mapsto(0,+\infty)$ given by
\begin{equation}\label{eq:spherical:density}
\varphi_n(z_1,\ldots,z_n)=c_n\frac{\prod_{i<j}|z_i-z_j|^2}{\prod_{i=1}^n(1+|z_i|^2)^{n+1}}
=c_n\mathrm{e}^{-2(n+1)\sum_{i=1}^nQ(|z_i|)}
\prod_{i<j}|z_i-z_j|^2,
\end{equation}
where $c_n$ is a normalizing constant. This is a planar Coulomb gas with rotationally invariant potential 
\begin{equation}
Q(|z|)=\frac{1}{2}\log(1+|z|^2).
\end{equation}
This is also a determinantal point process of $n$ particles on $\mathbb{C}$ 
endowed with the Lebesgue measure with kernel\footnote{Contrary to \cite{MR2552864}, we use here as a background the uniform distribution on the sphere in stereographic coordinates.}
\begin{equation}\label{eq:spherical:kernel}
    (z,w)\in\mathbb{C}^2\mapsto
    K_n(z,w)=\sqrt{\kappa(z)\kappa(w)}
    n\Bigr(\frac{1+z\overline{w}}{\sqrt{(1+|z|^2)(1+|w|^2)}}\Bigr)^{n-1}
    \quad\text{where}\quad
    \kappa(z)=\frac{1}{\pi(1+|z|^2)^2}.
\end{equation}
This means that the joint density in \eqref{eq:spherical:density} is given by
\begin{equation}
\varphi_n(z_1,\ldots,z_n)=\frac{1}{n!}\det\bigr[K_n(z_i,z_j)\bigr]_{1\leq i,j\leq n}.
\end{equation}
The function $\kappa$ in \eqref{eq:spherical:kernel} is a density on $\mathbb{C}$ with respect to the Lebesgue measure. The probability measure 
\begin{equation}
    \dd\nu(z)=\kappa(z)\dd^2z
\end{equation}
is known as a bivariate Student t distribution in Statistics and as a Barenblatt profile in Analysis of PDE (fast diffusion equation\footnote{See for instance \cite{cutfast} and references therein.}). 
It is heavy tailed. It is the image by the (north pole) stereographic projection \begin{equation}
    T:\mathbb{S}^2\subset\mathbb{R}^3\to\mathbb{C}\cup\{\infty\}
\end{equation}
of the uniform probability measure on the sphere $\mathbb{S}^2$. 
More precisely, for all $x=(x_1,x_2,x_3)\in\mathbb{S}^2\setminus\{(0,0,1)\}$ and $z\in\mathbb{C}$,
\begin{equation}\label{eq:T}
    T(x)=\frac{x_1+\mathrm{i}x_2}{1-x_3}
    \quad\text{and}\quad
    T^{-1}(z)
    =\frac{(2 \Re z, 2 \Im z,|z|^2-1)}{|z|^2+1}.
\end{equation}
In other words, this measure is the uniform probability distribution on the sphere written in stereographic coordinates. Furthermore, the image of the spherical ensemble by $T^{-1}$ is the gas on $\mathbb {S}^2$ with density with respect to the
uniform measure given, up to a multiplicative normalizing constant, by
\begin{equation}\label{eq:spherical:density:sphere}
(x_1,\ldots,x_n)\in(\mathbb{S}^2)^n
\mapsto\prod_{i<j}\|x_i-x_j\|^2_{\mathbb{R}^3},
\end{equation}
hence the name of the ensemble. This can also be seen as a two-dimensional analogue of the circular unitary ensemble (CUE). For all these aspects, we refer to \cite[Sec.~4.3.8]{MR2552864}. This point process on the $\mathbb{S}^2$ is invariant by isometries,
which allows us to study
the behavior near $\infty$ by dealing with
the behavior near $0$. This symmetry is a key ingredient of our proof of Theorem \ref{th:spherical:girko:rho}. 

When $R$ runs over the rotations of $\mathbb{S}^2$, then $T\circ R\circ T^{-1}$ runs over the maps of $\mathbb{C}\cup\{\infty\}$ of the form
\begin{equation}    \label{eq:su}
z\mapsto\frac{\alpha z+\beta}{-\overline{\beta}z+\overline{\alpha}}, \quad (\alpha,\beta)\in\mathbb{C}^2\setminus \{(0,0)\}.
\end{equation}
The invariance by scaling allows to take $|\alpha|^2+|\beta|^2=1$.
More information may be found in \cite{zbMATH07645442,zbMATH03176754}. 
As a consequence, 
if $Z$ is the planar point process or Coulomb gas \eqref{eq:spherical:density}, then for every
$(\alpha,\beta)$ as above,
\begin{equation} \label{eq:equalitySU}
    \frac{\alpha Z+\beta}{-\overline{\beta}Z+\overline{\alpha}}
    \overset{\dd}{=}
    Z.
\end{equation}
It is convenient to think of
$\mathbb C \cup \{\infty\}$ as
the projective line
$\mathbb CP^1 =
(\mathbb C^2\setminus \{(0,0)\})
/\sim$, where the relation $\sim$ defining the quotient is given by $(z_1,z_2) \sim (w_1,w_2)$
if they are colinear : the equivalence class $[(z,w)]$ 
should be thought of as the 
complex line passing
through the origin and $(z,w)$. A usual way to 
identify them is via the map
$z \mapsto [(z,1)]$ which takes the 
metric $(1+|z|^2)^{-2} |\mathrm d z|^2$
to the Fubini--Study metric\footnote{We
quickly recall a definition of
the Fubini--Study metric. Fix $\ell = [(z,w)]$ thought
of as a subspace of $\mathbb C^2$ and
let us define a Hermitian metric on 
the tangent space $T_\ell \mathbb CP^1$.
Noticing that $T_\ell \mathbb CP^1$ is
canonically identified
with the space of linear maps from
$\ell$ to $\mathbb C^2/\ell$, the
Hermitian metric on $\mathbb C^2$
induces one on $T_\ell \mathbb CP^1$. For another standard definition
using a local potential see 
\cite[Example 3.1.9]{fubini}.},
the metric of
$\mathbb CP^1$ induced from the Hermitian product
of $\mathbb C^2$.
In this way, the invariance under
rotations jumps out at us.
Take any Hermitian product on 
$\mathbb C^{n \times n}$
and endow $(\mathbb C^{n \times n})^2$
with the product Hermitian product or,
what is the same, 
endow 
$\mathbb C^{n \times n} \otimes \mathbb C^2$
with the tensor product of 
the Hermitian product of $\mathbb C^{n \times n}$
and the standard one of $\mathbb C^2$.
Then, the map
$(\mathbb C^{n \times n})^2
\to (\mathbb C^{n \times n})^2$,
\[(X,Y) \mapsto (\alpha X + \beta Y, 
-\overline \beta X + \overline \alpha Y),\]
preserves the Hermitian product of
$(\mathbb C^{n \times n})^2$ so that
if $(X,Y)$ follows a standard Gaussian law
in $(\mathbb C^{n \times n})^2$, its image
$(\alpha X + \beta Y, 
-\overline \beta X + \overline \alpha Y)$ also does.
This implies that
$XY^{-1} \overset{\dd}{=}
(\alpha X + \beta Y)(-\overline \beta X + \overline \alpha Y)^{-1} = 
(\alpha XY^{-1} + \beta)(-\overline \beta X Y^{-1}
+ \overline \alpha )^{-1}$.

\subsection{Kostlan observation and spectral radius}

Following \cite{MR1148410,MR2552864}, if $Z_n=(Z_{n,1},\ldots,Z_{n,n})$ is a random vector of $\mathbb{C}^n$ with density \eqref{eq:spherical:density}, then the determinantal structure and rotational invariance in \eqref{eq:spherical:density} give 
\begin{equation}\label{eq:kostlan:spherical}
\{|Z_{n,1}|,\ldots,|Z_{n,n}|\}\overset{\dd}{=}\{\xi_{n,1},\ldots,\xi_{n,n}\}
\end{equation}
where $\xi_{n,1},\ldots,\xi_{n,n}$ are independent and $\xi_{n,k}$ has density proportional to 
\begin{equation}
    x\geq0\mapsto x^{2k-1}\mathrm{e}^{-2(n+1)Q(x)}=\frac{x^{2k-1}}{(1+x^2)^{n+1}}.
\end{equation}
It follows that the random variable $\xi_{n,k}^2$ has density proportional to
\begin{equation}\label{eq:xi2}
  x\geq0\mapsto\frac{x^{k-1}}{(1+x)^{n+1}}.
\end{equation}
We recognize a Beta prime (or inverse Beta or Beta of the second kind)
distribution of density
\begin{equation}
  x\geq0\mapsto\frac{\Gamma(a+b)}{\Gamma(a)\Gamma(b)}\frac{x^{a-1}}{(1+x)^{a+b}},
\end{equation}
with $a=k$ and $b=n-k+1$, which is also the law of $B/(1-B)$ when
$B\sim\mathrm{Beta}(a,b)$, and also the law of $G_a/G_b$ were
$G_a\sim\mathrm{Gamma}(a,\lambda)$ and $G_b\sim\mathrm{Gamma}(b,\lambda)$ are
independent and $\lambda>0$ is an arbitrary scale parameter.

Moreover, we also note at this step that
$n/\xi_{n,n-k+1}^2\overset{\mathrm{d}}{\to}\mathrm{Gamma}(k,1)$ as $n\to\infty$.
Following \cite{MR3215627,MR3615091}, it is possible to get
 from \eqref{eq:kostlan:spherical} the asymptotic fluctuations of the
 spectral radius, namely
\begin{equation}\label{eq:spherical:rho}
    \frac{1}{\sqrt{n}}\rho(\MG)
    =\frac{1}{\sqrt{n}}\max_{1\leq k\leq n}|Z_{n,k}|
    \overset{\mathrm{d}}{=}
    \frac{1}{\sqrt{n}}\max_{1\leq k\leq n}\xi_{n,k}
    \xrightarrow[n\to\infty]{\mathrm{d}}
    \mathrm{Law}\Bigr(\max_{k\geq1}\frac{1}{\sqrt{\gamma_k}}\Bigr)
\end{equation}
where ${(\gamma_k)}_{k\geq1}$ are independent with $\gamma_k\sim\mathrm{Gamma}(k,1)$.
This is an alternative approach to the one using \eqref{eq:kostlan:gininf}.

\subsection{Equilibrium measure}

Regarding high dimensional asymptotic analysis, almost surely,
\begin{equation}\label{eq:spherical:mu}
\frac{1}{n}\sum_{k=1}^n\delta_{Z_{n,k}}
\xrightarrow[n\to\infty]{\text{weak}}
\frac{\Delta Q(\left|\cdot\right|)}{2\pi}\dd^2z
=\frac{\dd^2z}{\pi(1+|z|^2)^2}
=\nu,
\end{equation}
see for instance \cite{MR2926763}. The average version of \eqref{eq:spherical:mu} can be checked using the logarithmic potential and remains valid non asymptotically. Indeed, for all $z\in\mathbb{C}$,
$\BG$ and $\AG-z\BG$ are correlated, but by Gaussianity,
\begin{equation}
(\MG-z)\BG
=\AG-z\BG
\overset{\dd}{=}
\sqrt{1+|z|^2}\AG.
\end{equation}
Therefore, for all $z\in\mathbb{C}$, in $[-\infty,+\infty)$,
\begin{equation}\label{eq:E:logdet:spherical}
\mathbb{E}\log|\det(\MG-z)|+\mathbb{E}\log|\det(\BG)|
=n\log\sqrt{1+|z|^2}+\mathbb{E}\log|\det(\AG)|.
\end{equation}
Finally, by applying the operator $\frac{1}{2\pi}\Delta$ in the sense of distributions, we obtain the stunning non-asymptotic formula
\begin{equation}\label{eq:EESD:spherical}
\mathbb{E}\Bigg[ \frac{1}{n}\sum_{k=1}^n\delta_{Z_{n,k}} \Bigg] =\frac{\Delta\log(1+|z|^2)}{4\pi}\dd^2z=\frac{\dd^2z}{\pi(1+|z|^2)^2}=\nu.
\end{equation}
Alternatively, we could use the fact that the uniform distribution on
$\mathbb{S}^2$ is the unique distribution on $\mathbb{S}^2$ invariant by all
rotations, together with the fact that its image by the stereographic
projection $T$ is precisely $\nu$.

This formula is well known; see, for instance, \cite{MR2489167}, \cite{zbMATH06425285}, and \cite{zbMATH06125015}. More generally, if $A$ and $B$ are Girko matrices, then $A-zB=\sqrt{1+|z|^2}C_z$ where $C_z$ is a Girko matrix, but with a law that depends on $z$ in general, except in the Gaussian case. Also the argument above works only in the Ginibre case.

\section{Proof of Theorem \ref{th:spherical:girko:gininf} and Theorem \ref{th:spherical:girko:rho}}
\label{se:proof:th:spherical:girko:gininf+rho}

The key idea in a nutshell is as follows: if $A$ and $B$ are independent and have the same law, then
\begin{equation}\label{eq:sphere:rho:inverse}
    M^{-1}\overset{\dd}{=}M,\quad\text{which gives}
    \quad\rho_{\max}(M)\overset{\dd}{=}\rho_{\max}(M^{-1})
    =\max_{\lambda\in\mathrm{spec}(M)}\frac{1}{|\lambda|}
    =\frac{1}{\min_{\lambda\in\mathrm{spec}(M)}|\lambda|}
    =\frac{1}{\mathrm{dist}(0,\mathrm{spec}(M))}.
\end{equation}
This remains valid when $A$ and $B$ do not have the same law, provided that in the right hand side of the equality in distribution symbols we modify $M$ by swapping the laws of $A$ and $B$, which is not a problem for our universality objective. In other words, the model $M$ is globally invariant in law by inversion. Geometrically, the inversion exchanges on the two-sphere the points $\infty$ and $0$, the large scale behavior with the local scale behavior at the origin.

\begin{proof}[Proof of Theorem \ref{th:spherical:girko:gininf}]

We will first study the case when $A$ and $B$ are Ginibre matrices
so that $M = M^{\mathrm{Gin}}$.
To simplify calculations we 
argue why it is enough to deal with the
origin.
Indeed,
for any $\lambda_0 \in \mathbb C$, we may
take a rotation $R$ of
$\mathbb S^2 \simeq \mathbb C \cup \{\infty\}$ that maps $0$ to $\lambda_0$.
For instance, 
we could choose the explicit rotation
\begin{equation} \label{eq:ExplicitRotation}
    R(z)=\frac{z+\lambda_0}{-\overline{\lambda_0}z+1}
\end{equation}
and let us denote $\mathrm d R_0$
the differential of $R$ at the point $0$.
If we use the explicit matrix from 
\eqref{eq:ExplicitRotation} we can directly
calculate
the complex derivative $R'(0) = 1+|\lambda_0|^2$
so that
$\mathrm d R_0 = (1+|\lambda_0|^2) 
{\mathrm{Id}}_{2} $.
If we prefer to use a general rotation $R$, 
we can show that 
$|\det(\mathrm d R_0)| = (1+ |\lambda_0|^2)^2$
 because it preserves the uniform measure on the
sphere $(1+|z|^2)^{-2}d^2 z$, and
using that
$\mathrm d R_0 $ is a multiplication
by a complex number of squared norm
$(1 + |\lambda_0|^2)^2$ we have
$\mathrm d R_0 = 
(1 + |\lambda_0|^2) \mathrm{e}^{\mathrm{i}\theta}$
for some real $\theta$.
Finally, by Lemma
\ref{lem:ChangeOfCoordinates},
if we consider $x_0=0$,
$\varphi_1 = R$
and $\varphi_2$ as the identity map, we get
that
    \[\lim_{n \to \infty}\{\sqrt n (R(\lambda) - \lambda_0):
    \lambda \mbox{ is an eigenvalue of }M^{\mathrm{Gin}}\}
    = \mathrm d R_0
    \big(\lim_{n \to \infty}\{\sqrt n \lambda:
    \lambda \mbox{ is an eigenvalue of }M^{\mathrm{Gin}}\} \big)
    =\mathrm d R_0(\mathrm{Gin}_{\infty})\]
assuming the theorem holds
for the origin.
Since $\mathrm{Gin}_{\infty}$
is invariant under a multiplication
by $\mathrm{e}^{i\theta}$, we can forget the factor
$\mathrm{e}^{i\theta}$ so that
$\mathrm d R_0(\mathrm{Gin}_{\infty})
=(1 + |\lambda_0|^2)\mathrm{Gin}_{\infty} $.
On the other hand,
since $\{R(\lambda):
\lambda \mbox{ is an eigenvalue of }M^{\mathrm{Gin}}\}$ has the same law
as
$\{\lambda: \lambda \mbox{ is an eigenvalue of }M^{\mathrm{Gin}}\}$,
we get that
    \[\lim_{n \to \infty}\{\sqrt n (\lambda - \lambda_0):
    \lambda \mbox{ is an eigenvalue of }M^{\mathrm{Gin}}\}
    =(1+|\lambda_0|^2)\mathrm{Gin}_{\infty}.\]


We suppose now that $\lambda_0=0$.
Let $K_n$ be the kernel of the spectrum of $\MG$ seen as a determinantal point process, given by \eqref{eq:spherical:kernel}. It follows that the spectrum of $\sqrt{n}\MG$ forms a determinantal point process on $\mathbb{C}$ with kernel
\begin{equation}
    \frac{1}{n}K_n\Bigr(\frac{z}{\sqrt{n}},\frac{w}{\sqrt{n}}\Bigr)
    =\frac{\Bigr(1+\frac{z\overline{w}}{n}\Bigr)^{n-1}}{\pi\Bigr(1+\frac{|z|^2}{n}\Bigr)^{\frac{n+1}{2}}\Bigr(1+\frac{|w|^2}{n}\Bigr)^{\frac{n+1}{2}}}
    \xrightarrow[n\to\infty]{}
    \frac{\mathrm{e}^{z\overline{w}}}{\pi\mathrm{e}^{\frac{1}{2}|z|^2}\mathrm{e}^{\frac{1}{2}|w|^2}}
    =\sqrt{\gamma(z)\gamma(w)}\mathrm{e}^{z\overline{w}}
    \quad\text{where}\quad
    \gamma(z)=\frac{\mathrm{e}^{-|z|^2}}{\pi}.
\end{equation}
We recognize on the right-hand side the kernel of the infinite Ginibre $\mathrm{Gin}_\infty$ determinantal point process. Since the convergence is uniform with respect to $z$ and $w$ on compact subsets of $\mathbb{C}^2$, it follows that the determinantal point processes converge in distribution by, for instance, 
\cite[Proposition 3.10]{MR2018415}. 
In particular,
\[\lim_{n \to \infty} \mathbb{E}\Bigr[F\Bigr( \int f(\sqrt n (z-z_0))\dd\mu_{\MG}(z)\Bigr)\Bigr] =
\mathbb{E}\Bigr[F\Bigr(\int f \dd\mu_{(1+|z_0|^2)\mathrm{Gin_{\infty}}} \Bigr)\Bigr]\]
for every bounded continuous function $F: \mathbb R \to \mathbb R$ 
and any compactly supported continuous function
$f:\mathbb C \to \mathbb R$.

Let us now consider general $A$, $B$ and $M = AB^{-1}$ as in the hypothesis of
Theorem \ref{th:spherical:girko:gininf}.
The replacement principle for local universality at any point,
Theorem \ref{th:replacement}, tells us that we also have
\[\lim_{n \to \infty} \mathbb{E}\Bigr[F\Bigr(\int f(\sqrt n (z-z_0))\dd\mu_{M}(z)\Bigr)\Bigr] =
\mathbb{E}\Bigr[F\Bigr(\int f \dd\mu_{(1+|z_0|^2)\mathrm{Gin_{\infty}}}
\Bigr)\Bigr]\]
but, for now, only for $F$ and $f$ smooth compactly supported
functions. 
We can now apply Lemmas 
\ref{lem:WeakSmoothCompactly}
and
\ref{lem:ProcessesSmoothCompactly} from the last section to the
sequence of point processes
$(\mathcal X_n)_{n \geq 1}$ given by
$\mathcal X_n = \{\sqrt n(\lambda - z_0) : \lambda \in \mathrm{spec}(M)\}$
and to the point process
$\mathcal X$ being the scaled Ginibre process
$(1+|z_0|^2) \G_{\infty}$.

\end{proof}

\begin{proof}[Proof of Theorem \ref{th:spherical:girko:rho}]

Let us first show the convergence of $\sqrt n\rho_{\mathrm{min}}(M)$.
Since the map that takes a point configuration on $\mathbb C$ to 
the point configuration made of their norms is continuous and the map
taking a point configuration on $[0,\infty)$ to
the minimum is also continuous, we get the convergence of
the distance to the origin combining with Theorem \ref{th:spherical:girko:gininf}. For a proof of the latter standard
assertion, we may see \cite[Lemma A.1]{zbMATH07493826}. 
Now an observation due to Kostlan, see \cite{MR1148410} and \cite[Th.~4.7.3 and Sec.~7.2]{MR2552864}, states that 
\begin{equation}\label{eq:kostlan:gininf}
\{|\lambda|:\lambda\in\mathrm{Gin}_\infty\}
\overset{\dd}{=}
\{\xi_k:k\geq1\}
\end{equation}
where ${(\xi_k)}_{k\geq1}$ are independent random variables with $\xi_k^2\sim\mathrm{Gamma}(k,1)$ for all $k\geq1$. 
If $\xi_{\min} = \min \{\xi_k: k \geq 1\}$,
\[\sqrt n \rho_{\min}(M) \xrightarrow[n \to \infty]{\mathrm{law}} \xi_{\min} \]
and we may calculate $\mathbb P (\xi_{\min} \geq r)$ as
\[\mathbb P (\xi_{\min} \geq r)
=\mathbb{P}\Bigr(\min_{k\geq1}\xi_k\geq r\Bigr)
=\prod_{k=1}^\infty\mathbb{P}(\xi_k\geq r)
=\prod_{k=1}^\infty(1-F_k)(r^2)
\quad\text{where}\quad
F_k(x)=\mathbb{P}(\xi_k^2<x).
\]
Furthermore, the Gamma--Poisson duality gives
$\mathbb{P}(\xi_k^2>x)=\mathbb{P}(\pi_x<k)=\mathrm{e}^{-x}\sum_{j=0}^{k-1}\frac{x^j}{j!}$ where $\pi_x\sim\mathrm{Poisson}(x)$.

Finally, for the convergence of $\rho_{\max}(M)/\sqrt n$,
we notice that
\[\rho_{\max}(AB^{-1}) = \frac{1}{\rho_{\min}(BA^{-1})}\]
so that it is a consequence of the convergence of the minimum
for $BA^{-1}$.

Alternatively, we could use the Kostlan observation on the spherical ensemble before passing to the limit, avoiding the usage of $\mathrm{Gin}_\infty$, see \eqref{eq:spherical:rho}.
\end{proof}

\section{Proof of Theorem \ref{th:replacement:simple}}
\label{se:proof:th:replacement:simple}

Theorem \ref{th:replacement:simple} is an immediate consequence of the following more general and technical result, used with the choice $f_n(z)=f(\sqrt{n}(z-z_0))$ for $f\in\mathcal{C}^\infty_c(\C\to\mathbb{R})$ and $z_0 \in \C$.

    
\begin{theorem}[Comparison theorem]\label{th:replacement}\ \\
     Let $C_0,C_1>0$ and $\nu>0$ be fixed and such that $\nu<10^{-4}c_0$ where $c_0$ is as in condition \textup{\ref{it:condmom4}}.\\
     Let ${(f_n)}$ be a sequence of functions such that 
     \begin{equation}\label{eq_f}
     f_n \in\mathcal{C}_c^{2}(\C\to\mathbb{R}),\qquad
     \int|\Delta_z f_n(z)| \dd^2 z  \leq n^{\nu}, \qquad \mathrm{supp}(f_n) \subset\{z\in\mathbb{C}:|z|\leq C_0\}.
     \end{equation}
     Let $F:\mathbb{R}\to\mathbb{R}$ be a function with uniformly bounded derivatives by $C_1$ up to fifth order, namely
     \begin{equation}\label{eq_F}
     \max_{1\leq k\leq 5}\sup_{x\in \R} |F^{(k)}(x)| \leq C_1.
     \end{equation}
     Then, for all  $0<\delta<c_0/400$, we have
    \[
    \Bigr|\mathbb{E}\Bigr[F\Bigr(  \int f_n(z)\dd\mu_M(z)\Bigr)\Bigr]
    -\mathbb{E}\Bigr[F\Bigr( \int f_n(z)\dd\mu_{\MG}(z)\Bigr)\Bigr]\Bigr| =O\big( n^{-\delta}\big).
    \]
\end{theorem}

Theorem \ref{th:replacement} proves universality in a broad sense, not only for gap distribution or spectral radius, but also for CLT with test functions satisfying \eqref{eq_f}. One can also use it to show the universality of convergence speed, as long as the optimal speed is much smaller than $n^{-\delta}$ for a sufficiently small $\delta>0$.





\bigskip

Before we present the proof of Theorem \ref{th:replacement}, we sketch an outline of it, in three steps:
\begin{enumerate}
	\item Via the Hermitization trick introduced by \cite{zbMATH03901742}, one reduces the spectral problem of $AB^{-1}$ near the point $z_0\in \C$ to studying the singular value statistics of $A-zB$ with $z$ near $z_0$, \ie
	\begin{equation}\label{relation}
		\mathcal{L}= \int f_n(x)\dd\mu_{AB^{-1}}(x)=\frac{1}{4\pi} \int \Delta_z f_n(z) \log \det [(A-zB) (A-zB)^*]\dd^2 z,
	\end{equation}
	 see more details in \eqref{detAB}-\eqref{eq_hermitise} below. In particular, assuming $B$ is invertible, $z$ is an eigenvalue of $AB^{-1}$ is equivalent to that the smallest singular value of $A-zB$ is zero. This translates the eigenvalue problem of the non-Hermitian matrix into that of the Hermitian model $(A-zB) (A-zB)^*$ near zero. 
	
	\item To regularize the quantity on the right side of \eqref{relation}, we introduce a parameter $\eta_0$ slightly below the typical size of the smallest singular value of $A-zB$ near zero and show that (see Lemma \ref{lemma_tiny})
	\begin{align}\label{approx}
		\mathcal{L} \approx \mathcal{L}^{\eta_0}=&\frac{1}{4\pi} \int \Delta_z f_n(z) \log \det \big[(A-zB) (A-zB)^*+\eta_0^2\big]\dd^2 z\nonumber\\
		\approx&-\frac{1}{4\pi} \int \Delta_z f_n(z) \int_{\eta_0}^{T} \Im \mathrm{Tr} G^{z}(\ii \eta) \dd \eta \dd^2 z, \qquad \qquad \eta_0=n^{-1-\epsilon}, \quad T=n^{100},
	\end{align}
     for a small $\epsilon>0$, where the approximation holds in the sense of absolute mean, and $G^{z}$ is the resolvent (or Green function) of the Hermitised matrix of $A-zB$ as defined in \eqref{eq_hermitise}-\eqref{eq_resolvent} below. This step relies on the effective lower bound of the smallest singular value of Girko matrices from 
     \cite{Davies} and the local law estimates for the Green function from \cite{zbMATH06347297}. It is worth noting that the polynomial lower bound for the smallest singular value obtained in \cite{MR2908617} and inspired from \cite{MR2684367} is not enough for our purposes regarding very tiny singular values, see \eqref{eq_tiny}.
   

	\item  From \eqref{approx},  it then suffices to show that under the fourth moment condition \ref{it:condmom4} (see Lemma \ref{lemma_GFT}), 
	$$\mathbb{E}[F(\mathcal{L}^{\eta_0})]=\mathbb{E}^{\mathrm{Gin}}[F(\mathcal{L}^{\eta_0})]+o(1),$$ 
	where $\mathbb{E}^{\mathrm{Gin}}$ indicates the corresponding expectation with $A$ and $B$ being Ginibre matrices. This step relies on the so-called long-time Green function comparison theorem (GFT). To achieve this, we let the matrix entries of $A$ and $B$ run independent Ornstein--Uhlenbeck dynamics. Then the corresponding matrix flow $A_t-zB_t$ interpolates between $A-zB$ and the Ginibre counterpart $A^{\mathrm{Gin}}-zB^{\mathrm{Gin}}$. We use Itô's formula and cumulant expansion formula to compute the time derivative of $\mathbb{E}[F(\mathcal{L}_t^{\eta_0})]$ and then use the local law estimates for the Green function 
    and the fourth moment condition \ref{it:condmom4} to show the smallness of $\dd\mathbb{E}[F(\mathcal{L}_t^{\eta_0})]$ along the interpolating flow. 
	
\end{enumerate}

\subsection{Proof of Theorem \ref{th:replacement}}
In this section, we normalize both Girko matrices $A$ and $B$ by $1/\sqrt{n}$, \ie $\A=A/\sqrt{n}$ and $\B=B/\sqrt{n}$, and their quotients are the same, $M=AB^{-1}=\A \B^{-1}$. We also set $\MG={\tilde A}^{\mathrm{Gin}}({\tilde B}^{\mathrm{Gin}})^{-1}=\AG (\BG)^{-1}$ with ${\tilde A}^{\mathrm{Gin}}=\AG/\sqrt{n}$, ${\tilde B}^{\mathrm{Gin}}=\BG/\sqrt{n}$ being normalized Ginibre matrices. For conventional consistency, we always write $\mathrm{Gin}$ as a superscript to denote the special Gaussian case. 

Using that $\Delta_z \log |z|=2\pi\delta_0$ and applying integration by parts twice, for any $f\in \mathcal{C}^{2}_c(\C \rightarrow \R)$, we have
\begin{align}\label{detAB}
     \int f_n(x)\dd\mu_M(x)=&\frac{1}{2\pi}\int \Delta_z f_n(z) \log |\det (\A \B^{-1}-z)|\dd^2 z\nonumber\\
    =&\frac{1}{2\pi}\int \Delta_z f_n(z) \big(\log |\det (\A-z\B)|-\log |\det (\B)|\big)\dd^2 z,
\end{align}
where we also used that $\det(\A \B^{-1}-z)=\det(\A-z\B)/\det(\B)$. Since $\log |\det \B|$ is $z$-independent, again using integration by parts in $z$, the last term in \eqref{detAB} vanishes. Hence we have
\begin{align}\label{X}
	 \int f_n(x)\dd\mu_M(x)=\frac{1}{2\pi} \int \Delta_z f_n(z) \log |\det (X^z) |\dd^2 z, \qquad  X^z=\A-z \B \in \C^{n \times n}.
\end{align}
By the Hermitization trick of \cite{zbMATH03901742}, we write
\begin{align}\label{eq_hermitise}
     \int f_n(x)\dd\mu_M(x)=\frac{1}{4\pi} \int \Delta_z f_n(z) \log |\det H^{z}|\dd^2 z, \qquad  H^{z}=\begin{pmatrix}
		0  &   X^z \\
		(X^z)^*    & 0
	\end{pmatrix}  \in \C^{2n \times 2n}.
\end{align}
  The eigenvalues of the Hermitised matrix $H^{z}$ are then denoted by $\big(\lambda^{z}_{\pm i}\big)_{i=1}^n$ with $\lambda^{z}_{-i}=-\lambda^{z}_i$ and $0\leq \lambda^{z}_i \leq \lambda^{z}_{i+1}$ for $1\leq i\leq n$. Note that the non-negative eigenvalues $\big(\lambda^{z}_{ i}\big)_{i=1}^n$ are also the singular values of $X^z=(A-zB)/\sqrt{n}$.  We further define the resolvent (or Green function) of $H^{z}~(z\in \mathbb{C})$ by
\begin{align}\label{eq_resolvent}
G^{z}(w)=(H^{z}-w)^{-1}, \qquad \quad w \in \mathbb{C} \setminus \mathbb{R}.
\end{align}
Using that $\partial_\eta \log(\lambda^2+\eta^2)=\frac{2\eta}{\lambda^2+\eta^2}$, we obtain from \eqref{eq_hermitise} that
\begin{align}
     \int f_n(x)\dd\mu_M(x)=-\frac{1}{4\pi} \int \Delta_z f_n(z) \int_{0}^{T} \Im \mathrm{Tr} G^{z}(\ii \eta) \dd \eta \dd^2 z+\frac{1}{4\pi}  \int \Delta_z f_n(z) \log |\det (H^{z}-\ii T)|\dd^2 z,
\end{align}
where we chose $T=n^{100}$. Note that 
	\begin{align}
		\log |\det (H^{z}-\ii T)|-2n \log T= \sum_{j} \log \left( 1+\Big(\frac{\lambda_j^{z}}{T}\Big)^2\right) \leq \frac{\mathrm{Tr} (H^{z})^2}{T^2} \prec n^{-199},
	\end{align}
where we used that $\log(1+|x|)\leq |x|$ and, for
the stochastic domination, we used 
\ref{it:condmom} to get that
\begin{equation}
\mathbb P(\mathrm{Tr}(X^{z})(X^{z})^* \geq n^\xi n)
= \mathbb P\Bigr(\sum_{i,j=1}^n |a_{ij} - z b_{ij}|^2
\geq n^\xi n^2\Bigr) \leq  \sum_{i,j=1}^n 
\mathbb P (|a_{ij} - z b_{ij}|^2 \geq n^\xi)
\leq 2^{k-1} n^2 n^{-k \xi} D_k
\end{equation}
which implies $\mathrm{Tr}(X^{z})(X^{z})^*\prec n$.
Recall that $f_n \in\mathcal{C}^2_c(\C)$ 
satisfies $\int |\Delta_z f_n(z)| \dd^2 z \leq n^{\nu}$ with $\nu<10^{-4}c_0$. Hence we have 
\begin{align}
     \int f_n(x)\dd\mu_M(x)=-\frac{1}{4\pi} \int \Delta_z f_n(z) \int_{0}^{T} \Im \mathrm{Tr} G^{z}(\ii \eta) \dd \eta \dd^2 z+O_\prec(n^{-100}). 
\end{align}
We next split the above $\eta$-integral into two parts, \ie 
\begin{align}\label{split_I}
    I^{\leq\eta_0}+\II=-\frac{1}{4\pi} \int \Delta_z f_n(z) \Big(  \int_{0}^{\eta_0}+\int_{\eta_0}^{T}\Big) \Im \mathrm{Tr} G^{z}(\ii \eta) \dd \eta \dd^2 z, \quad\qquad \eta_0=n^{-1-\epsilon},
\end{align}
 for some small $\epsilon>0$ to be fixed later. The lemma below shows that the first $\eta$-integral $I^{\leq\eta_0}$ is negligible in the sense of absolute mean.

\begin{lemma}\label{lemma_tiny}
Fix any small $\epsilon>0$ and set $\eta_0=n^{-1-\epsilon}$. Then the following holds
$$\mathbb{E}\big|I^{\leq\eta_0}\big|=O\big( n^{-\epsilon/2+\nu}\big).$$
\end{lemma}

Recall that $c_0>0$ is a small constant in \ref{it:condmom4} and $\nu<10^{-4}c_0$ from \eqref{eq_f}. Hence we can choose a small $\epsilon>0$ such that $100\nu<\epsilon<c_0/100$, where the upper restriction of $\epsilon$ will be used in Lemma \ref{lemma_GFT} below. Note that $F$ has a bounded derivative as in \eqref{eq_F}. Then by a simple Taylor expansion and using Lemma \ref{lemma_tiny}  with $\nu<\epsilon/100$, we have
\begin{align}\label{eq_quantity}
	\mathbb{E} \Big[ F \Big(  \int f_n(x)\dd\mu_M(x)\Big) \Big]=\mathbb{E} \big[ F (\II) \big]+O\big( n^{-\epsilon/4}\big).
\end{align}
Since the Ginibre matrix is also a Girko matrix, the above also applies to $\MG=\AG(\BG)^{-1}$ with $\A$ and $\B$ replaced by $\AG$ and $\BG$. 
It then suffices to compare the expectation $\mathbb{E}\big[ F (\II) \big]$ with its Ginibre counterpart $\mathbb{E}^{\G}\big[ F (\II) \big]$.
\begin{lemma}\label{lemma_GFT}
Recall that $\nu<10^{-4}c_0$ from \eqref{eq_f}.	Fix any $0<\epsilon<c_0/100$ and let $\eta_0=n^{-1-\epsilon}$.  Then 
	\begin{align}
    \label{eq_GFT}
		\Big|\mathbb{E} \big[ F (\II) \big] -\mathbb{E}^{\mathrm{Gin}}  \big[ F (\II) \big] \Big|=O(n^{-c_0/4}),
	\end{align}
where we used $\mathbb{E}^{\mathrm{Gin}} $ to denote the corresponding expectation with $\A$ and $\B$ being Ginibre matrices. 
\end{lemma}

 Therefore, Theorem \ref{th:replacement} follows directly from \eqref{eq_quantity} for both matrices $M$ and $\MG$ in combination with Lemma~\ref{lemma_GFT} for $\epsilon<c_0/100$, where $c_0>0$ is small constant from \ref{it:condmom4}.
 
 \medskip
 
 The remaining part of this section is to prove Lemma \ref{lemma_tiny} and Lemma \ref{lemma_GFT}.

\begin{proof}[Proof of Lemma \ref{lemma_tiny}]
	Recall that $H^{z}$ in \eqref{eq_hermitise} is the Hermitised matrix of $X^z=\wt{A}-z\wt{B}$, and $G^{z}$ is its resolvent defined in \eqref{eq_resolvent}. For each fixed $z\in \C$,  $X^z=(x_{ij})$ is still a Girko matrix with $\mathbb{E}[x_{ij}]=0$, $\mathbb{E}[(x_{ij})^2]=0$ and $\mathbb{E}[|x_{ij}|^2]=(1+|z|^2)/n$. In particular if we set $z=0$, then $X^{z=0}$ becomes a standard Girko matrix with variance $n^{-1}$.  From \cite[Theorem 1.1]{zbMATH06347297} (see also \cite{zbMATH06221300,zbMATH06261246,zbMATH06330939}), fixing any small $\xi>0$, the following local laws for the resolvent $G^{z}$ with $z=0$ 
	\begin{align}\label{eq:local_law_G}
		\Big| \big\langle \mathbf{x}, \big(G^{z=0}(w)-m_{sc}(w) \big)\mathbf{y} \big\rangle \Big|\prec \sqrt{
			\frac{1}{n \eta}}, \qquad \Big|\frac{1}{2n}  \mathrm{Tr}G^{z=0}(w)-m_{sc}(w)\Big|  \prec \frac{1}{n \eta},
	\end{align}
	hold uniformly for any $|w|\leq 10$ and $\eta=|\Im w| \geq n^{-1+\xi}$ and for any deterministic unit vectors $\mathbf{x}, \mathbf{y}\in \mathbb{C}^{2n}$, where $m_{sc}$ is the Cauchy--Stieltjes transform of the semicircle distribution $\rho_{sc}(x)=\frac{1}{2 \pi} \sqrt{(4-x^2)_+}$.  Note that for each fixed $z\in \C$, the independence of $A$ and $B$
    implies that $H^z$ 
    has the same distribution as $\sqrt{1+|z|^2} \tilde{H}$, where $\tilde{H}$ is the Hermitised matrix obtained from a standard Girko matrix with variance $1/n$ (note that the density of $\tilde H$ could depend on $z$). Thus, the resolvent of $\tilde{H}$ satisfies the same local law as in \eqref{eq:local_law_G}. A simple rescaling then yields, for each fixed $z\in \C$, the following local laws for $G^z$, 
	\begin{align}\label{eq:local_law_G_z}
		\Big| \big\langle \mathbf{x}, \big(G^z(w)-m^z(w)\big) \mathbf{y} \big\rangle \Big|  \prec \sqrt{
			\frac{1}{n \eta}}, \qquad \Big|\frac{1}{2n}  \mathrm{Tr}G^{z}(w)-m^z(w)\Big|  \prec \frac{1}{n\eta},
	\end{align}
	hold uniformly for any $|w|\leq 10$ and $\eta=|\Im w| \geq n^{-1+\xi}$, with the deterministic function $m^{z}$ given by
	\begin{align}\label{rho}
		m^{z}(w)=\frac{1}{\sqrt{1+|z|^2}} m_{sc}\Big( \frac{w}{\sqrt{1+|z|^2}}\Big), \qquad w\in \C \setminus \R.
	\end{align}
It is easy to check that $|m^{z}(w)| \leq 1$ uniformly for any $w \in \C \setminus \R$ and $z\in \C$. Note that $G^{z}$ is Lipschitz continuous in $z$. Though the estimates in \eqref{eq:local_law_G_z} hold with very high probability for each fixed $z$, they can be extended to very high probability estimates simultaneously for all $|z|\leq C_0$ by a classic grid argument and union bound (see, e.g., below Eq. (7.55) of \cite{zbMATH06780221}).  More precisely, divide the domain $\Lambda:=\{z\in \C: |z|\leq C_0\}$ into equal grids of size $n^{-100}$ and set $\Lambda_n:=n^{-100}\mathbb{Z}^2 \cap \Lambda$. From  the definition of $\prec$ in \eqref{prec}, fixing any small $\xi'>0$ and large $D>100$, we have 
    \begin{align}\label{union}
    &\mathbb{P}\left( \exists z_j \in  \Lambda_n  : \sup_{\|\bf{x}\|=\|\bf{y}\|=1} \sup_{|w|\leq 10, \eta \geq n^{-1+\xi}} \Big| \big\langle \mathbf{x}, \big(G^{z_j}(w)-m^{z_j}(w)\big) \mathbf{y} \big\rangle \Big|  \geq \frac{n^{\xi'}}{\sqrt{n\eta}}  \right) \\
    & \qquad \qquad \leq \sum_{z_j \in \Lambda_n}\mathbb{P}\Big(  \sup_{\|\bf{x}\|=\|\bf{y}\|=1} \sup_{|w|\leq 10, \eta \geq n^{-1+\xi}} \Big| \big\langle \mathbf{x}, \big(G^{z_j}(w)-m^{z_j}(w)\big) \mathbf{y} \big\rangle \Big|  \geq \frac{n^{\xi'}}{\sqrt{n\eta}} \Big) \leq n^{-D+200}.
    \end{align}
    By choosing $D$ sufficiently large, this proves the first estimate in \eqref{eq:local_law_G_z} for all  the lattice points $z_j \in \Lambda_n$ simultaneously. Note that from the resolvent identity and that $\|G^z(w)\| \leq \eta^{-1}$, we have, for any $|z-z'| \leq n^{-100}$,
    \begin{align}
    \Big| \big\langle \mathbf{x}, \big(G^{z}(w)-G^{z'}(w)\big) \mathbf{y} \big\rangle \Big|  \leq \|G^{z}(w)-G^{z'}(w)\| \leq  |z-z'| \|G^{z}(w)\| \|G^{z'}(w)\| \prec n^{-10},
    \end{align}
    holds uniformly for any $|w|\leq 10$ and $\eta=|\Im w| \geq n^{-1+\xi}$ and deterministic unit vectors $\bf{x}, \bf{y}$.  Combining with the deterministic bounds for $|m^{z}-m^{z'}|$ from \eqref{rho}, we have obtained the first estimate in \eqref{eq:local_law_G_z} simultaneously for all $|z|\leq C_0$. The same also applies to the second estimate in \eqref{eq:local_law_G_z}. From now on, without specific mention, all the estimates below hold simultaneously for all $|z|\leq C_0$.  

	Now we are ready to prove Lemma \ref{lemma_tiny}. The argument is similar to the proof of \cite[Lemma 4]{zbMATH07329432}. By a simple spectral decomposition of $H^{z}$, we have
	\begin{align}\label{int_small}
		I^{\leq\eta_0} = -\frac{1}{2\pi} \int \Delta_z f_n(z) \Big[\int_{0}^{\eta_0} \sum_{j=1}^{n} \frac{\eta}{(\lambda_j^z)^2+\eta^2} \dd \eta \Big] \dd^2 z, \qquad \eta_0=n^{-1-\epsilon},
	\end{align}
where $\big(\lambda^{z}_{j}\big)_{j=1}^n$ are the ordered non-negative eigenvalues of $H^{z}$, that are also singular values of $X^z$.
	Let $L>0$ be a large number to be fixed later.  Set $\eta_L=n^{-L}$ and $\eta_2=n^{-1-\epsilon/2}$. Taking the expectation, we split the above sum of eigenvalues into three parts as below:
	\begin{align}\label{eq_split}
		\mathbb{E} \Big[\int_{0}^{\eta_0} \sum_{j=1}^{n} \frac{\eta}{(\lambda^z_j)^2+\eta^2} \dd \eta \Big]=&\mathbb{E} \Big[\int_{0}^{\eta_0} \sum_{\lambda^z_j \geq \eta_2} \frac{\eta}{(\lambda_j^z)^2+\eta^2} \dd \eta \Big]\nonumber\\
		&+\mathbb{E}\Big[ \Big(\sum_{\eta_L\leq \lambda^z_j \leq \eta_2} +\sum_{\lambda^z_j\leq \eta_L}\Big) \log \Big( 1+\frac{\eta_0^2}{(\lambda^z_j)^2}\Big)\Big],
	\end{align}
	where we also used that $\partial_\eta \log(\lambda^2+\eta^2)=\frac{2\eta}{\lambda^2+\eta^2}$ in the last line. For the first line of \eqref{eq_split}, we have
	\begin{align}
		\mathbb{E} \Big[\int_{0}^{\eta_0} \sum_{\lambda_j \geq \eta_2} \frac{\eta}{\lambda_j^2+\eta^2} \dd \eta  \Big]
		\leq &2\mathbb{E} \Big[\int_{0}^{\eta_0} \sum_{\lambda_j\geq \eta_2} \frac{\eta}{\lambda_j^2+\eta_2^2} \dd \eta  \Big]\nonumber\\
		\leq & \mathbb{E} \Big[ \sum_{j=1}^{n} \frac{\eta_0^2}{(\lambda^z_j)^2+\eta_2^2}   \Big]= \frac{\eta_0^2}{2\eta_2} \mathbb{E} \big[\Im\mathrm{Tr}G^{z}(\ii \eta_2)\big].
	\end{align}
	Note that the function $\eta\Im\mathrm{Tr}G^{z}(\ii \eta)$ is non-decreasing in $\eta>0$. Choosing  $\eta_3=n^{-1+\xi}$
    for any fixed $\xi \in (0,\epsilon/2)$,
	\begin{align}\label{eq_last}
		\mathbb{E} \Big[\int_{0}^{\eta_0} \sum_{\lambda_j\geq \eta_2} \frac{\eta}{\lambda_j^2+\eta^2} \dd \eta  \Big] \leq
		\frac{\eta_0^2}{\eta_2} \frac{\eta_3}{\eta_2} \mathbb{E} \big[\Im\mathrm{Tr}G^{z}(\ii \eta_3)\big] =O (n^{-\epsilon/2}),
	\end{align}
	where we used the local law in \eqref{eq:local_law_G_z}  for $\eta_3=n^{-1+\xi}$
    and that $m^z$ from \eqref{rho} is bounded.

	We next estimate the last line of \eqref{eq_split}. Note that the non-negative eigenvalues $\big(\lambda^{z}_{ i}\big)_{i=1}^n$ of $H^{z}$ coincide with the singular values of $X^z=(A-zB)/\sqrt{n}$, where $A$ and $B$ are independent Girko matrices satisfying \ref{it:condensity}-\ref{it:condmom}. 
     Recall from \cite[Lemma 3.4]{Davies} (see also \cite{invert,wegner}) 
    that there exists a constant $C$ depending on $C_0$ and $D_0$ from \ref{it:condensity} such that the smallest singular value $\lambda_1^z$ of $(A-zB)/\sqrt{n}$ with $|z|\leq C_0$ satisfies
	\begin{align}\label{eq:tail_estimate}
		\mathbb{P} \big(\lambda^z_1 \leq t \big) \leq C n^2 t^2, \qquad \forall t>0,\quad |z|\leq C_0.
	\end{align}
 Note that for $\eta_2=n^{-1-\epsilon/2}$ and $\eta_3=n^{-1+\xi'}$ for any small $\xi'>0$
	$$ \frac{\eta_3}{\eta_2} \Im\mathrm{Tr}G^{z}(\ii \eta_3) \geq \Im\mathrm{Tr}G^{z}(\ii \eta_2)=2\sum_{j=1}^{n} \frac{\eta_2}{(\lambda^z_j)^2+\eta_2^2}  \geq \frac{2}{\eta_2} \# \big\{ 1\leq j\leq n: \lambda_j^z \leq \eta_2 \big\}.$$
	Hence using the local law in \eqref{eq:local_law_G_z}, by the definition of stochastic domination in \eqref{prec}, we have
	\begin{align}\label{eq:rigidity}
		\theta:=\#\big\{j\in \llbracket n
        \rrbracket: \lambda^z_j \leq n^{-1} \big\} \leq n^{\xi'},
	\end{align}
    for any small $\xi'>0$, with very high probability. Using \eqref{eq:tail_estimate} for $t=\eta_2=n^{-1-\epsilon/2}$ and \eqref{eq:rigidity} for $\xi'<\epsilon/100$, the first sum in the last line of \eqref{eq_split} is bounded by
	\begin{align}\label{eq_middle}
		\mathbb{E}\Big[ \sum_{\eta_L \leq  \lambda^z_j \leq \eta_2} \log \Big( 1+\frac{\eta_0^2}{(\lambda^z_j)^2}\Big)\Big] &\leq 
        \mathbb E \Big[
        \theta \log \Big( 1+\frac{\eta_0^2}{\eta_L^2}\Big)
        1_{\lambda_1^z \leq \eta_2}   
        \Big]\\
        &\leq
        n^{\xi'} C_L (\log n)\mathbb{P}\big( \lambda^z_1 \leq \eta_2 \big) =O(n^{-\epsilon/2}),
	\end{align}
where we have used
that, for every $D>0$, the inequality
$\mathbb E[\theta 1_{\theta > n^{\xi'}}]
\leq 
n\mathbb E[1_{\theta > n^{\xi'}}] 
\leq n^{-D}$ holds for $n$ large enough.
    For the last part in \eqref{eq_split} with tiny eigenvalues $\lambda^z_j \leq \eta_L=n^{-L}$, using that $\log ((\lambda^z_j)^2 + (\eta_0)^2) \leq 0$ 
    for $\lambda^z_j \leq \eta_L$ and the tail bound for the smallest singular value in \eqref{eq:tail_estimate}, we obtain
    \begin{align}\label{eq_tiny}
		\mathbb{E}\Big[ \sum_{ \lambda^z_j \leq \eta_L} \log \Big( 1+\frac{\eta_0^2}{(\lambda^z_j)^2}\Big)\Big] \leq & 2 \mathbb{E}\Big[ \sum_{ \lambda^z_j \leq \eta_L} -\log (\lambda^z_j)\Big] \leq 
        2n \mathbb{E} \Big[-\log \big(\lambda^z_1\big) \mathbf{1}_{\lambda^z_1 \leq \eta_L}  \Big]\nonumber\\
		= &  2n \bigg(
        L \log n \mathbb P \big(\lambda^z_1 
        \leq n^{-L} \big) + 
        \int_{L \log n}^{\infty} \mathbb{P}\big(\lambda^z_1 \leq \mathrm{e}^{-s} \big)\dd s \bigg) =O(n^{-100}),
	\end{align}
    by choosing a sufficiently large $L>0$ using \eqref{eq:tail_estimate}.
  
	Summing up \eqref{eq_last}, \eqref{eq_middle} and \eqref{eq_tiny}, we conclude from \eqref{int_small} and \eqref{eq_split} that
	\begin{align}
		\mathbb{E}| I^{\leq\eta_0}| \leq \frac{1}{2\pi} \int \big|\Delta_z f_n(z) \big| \mathbb{E} \Big[\int_{0}^{\eta_0} \sum_{j=1}^{n} \frac{\eta}{(\lambda_j^z)^2+\eta^2} \dd \eta \Big] \dd^2 z =O\big( n^{-\epsilon/2+\nu}\big), 
	\end{align}
	 where we also used that $\int|\Delta_z f_n(z)| \dd^2 z  \leq n^{\nu}$. This finished the proof of Lemma \ref{lemma_tiny}.

\end{proof}

\begin{proof}[Proof of Lemma \ref{lemma_GFT}]
	Define the following the Ornstein--Uhlenbeck matrix flows
	\begin{equation}\label{OU}
		\dd  \A_t=-\frac{1}{2} \A_t \dd t+\frac{\dd  \mathcal{B}^{(1)}_t}{\sqrt{n}}, \qquad \dd  \B_t=-\frac{1}{2} \B_t \dd  t+\frac{\dd  \mathcal{B}^{(2)}_t}{\sqrt{n}}, 
	\end{equation}
	with initial ensembles $\A_0=\A=A/\sqrt{n}$ and $\B_0=\B=B/\sqrt{n}$, where $A,B$ are independent Girko matrices satisfying \ref{it:condensity} --\ref{it:condmom}, and $\mathcal{B}^{(1)}_t$ and $\mathcal{B}^{(2)}_t$ are two independent $n\times n$ matrices whose entries are i.i.d. standard complex Brownian motions. Then for any $z \in \C$ we set $X^z_t=\A_t-z \B_t$ as in \eqref{X}, use $X^z_t$ to construct the Hermitised matrix $H^z_t$ as in \eqref{eq_hermitise} and define the corresponding Green function as in \eqref{eq_resolvent}, \ie
\begin{align}\label{xh}
    H^{z}_t=\begin{pmatrix}
		0  &   X_t^z \\
		(X_t^z)^*    & 0
	\end{pmatrix}, \quad\qquad G^z_t(w)=(H^z_t-w)^{-1}, \qquad w \in \C \setminus \R, \quad z\in \C,\quad t\geq 0.
\end{align}
We additionally introduce the following notations:
\begin{align}\label{def_HAB}
H_t^{(1)}=\begin{pmatrix}
		0  &   \A_t \\
		(\A_t)^*    & 0
	\end{pmatrix}, \qquad H_t^{(2)}=\begin{pmatrix}
		0  &   \B_t \\
		(\B_t)^*    & 0
	\end{pmatrix}, \qquad \dd\mathcal{M}^{(l)}_t=\begin{pmatrix}
		0  &   \dd\mathcal{B}^{(l)}_t \\
		\big(\dd\mathcal{B}^{(l)}_t\big)^*    & 0
	\end{pmatrix}, \qquad l=1,2.
\end{align}
For notational simplicity, we may often drop the dependence $t$ and write $H^{(l)}_t=(h^{(l)}_{ij})_{i,j\in 
\llbracket 2n \rrbracket}$ for $l=1,2$.
Note that the non-zero entries $(h^{(l)}_{aB})$ for $a\in \llbracket n \rrbracket$ and 
$B\in \llbracket n+1,2n \rrbracket$ are independent random variables that depend on time.

Recall the definition of $\II$ in \eqref{split_I}. Note that $\mathrm{Tr} G_t^{z}(\ii \eta)$ is purely imaginary due to the block structure of $H_t^z$ in \eqref{xh}, so we have $\mathrm{Tr} G_t^z(\ii \eta)=\ii\Im \mathrm{Tr} G_t^z(\ii \eta)$. Then we define the following short-hand
\begin{align}\label{def_L}
	\II_t=\frac{\ii}{4\pi} \int \Delta_z f_n(z) \int_{\eta_0}^{T} \mathrm{Tr} G^z_t(\ii \eta) \dd \eta \dd^2 z,  
\end{align}
with $T=n^{100}$ and $\eta_0=n^{-1-\epsilon}$ for $0<\epsilon<c_0/100$. We then use the matrix flows \eqref{OU} and Itô's formula to derive the dynamics of $G^{z}_t$, and hence $\II_t$, as a function of $\big\{h^{(l)}_{aB}\big\}$ in \eqref{def_HAB} for $a\in 
\llbracket n \rrbracket$ and $B\in \llbracket n+1,2n \rrbracket$. More precisely, 
\begin{align}
\dd \II_t=& \Theta \dd t+  \dd M_t,
\end{align}
with the drift function $\Theta$ given by
\begin{align}\label{eq_drift}
\Theta =   \sum_{a=1}^n \sum_{B=n+1}^{2n} \sum_{l=1}^2 \left(-\frac{1}{2}  h^{(l)}_{aB} \frac{  \partial \II_t}{\partial h^{(l)}_{aB}}  -\frac{1}{2} \overline{h^{(l)}_{aB}} \frac{\partial \II_t}{\partial \overline{h^{(l)}_{aB}}}  + \frac{1}{n} \frac{\partial^2 \II_t}{\partial h^{(l)}_{aB} \partial  \overline{h^{(l)}_{aB}}} \right)
=:&\Theta_1+\Theta_2+\Theta_3,
\end{align}
and the diffusion term $\dd M_t$ given by
\begin{align}\label{eq_diffusion}
\dd M_t =&\sum_{a=1}^n \sum_{B=n+1}^{2n} \sum_{l=1}^2  \left(\frac{ \partial \II_t}{\partial h^{(l)}_{aB}}  \frac{\dd\mathcal{M}^{(l)}_{aB}}{\sqrt{n}} + \frac{\partial \II_t}{\partial \overline{h^{(l)}_{aB}}} \frac{\big(\dd\mathcal{M}^{(l)}_{aB}\big)^*}{\sqrt{n}}  \right).
\end{align}
We point out that $I_t^{\geq \eta_0}$ is viewed
as a function of $(\widetilde A_t, \widetilde B_t)$,
and the derivatives are taken with respect
to these variables.
Note that the differentiation rules below are obtained from the definition of resolvent in \eqref{xh} and \eqref{def_HAB}:
	\begin{align}
    &\frac{\partial (G^z_t)_{ij}}{\partial h^{(1)}_{aB}}=-(G^z_t)_{ia} (G^z_t)_{Bj}, \qquad \frac{\partial (G^z_t)_{ij}}{\partial \overline{h^{(1)}_{aB}}}=-(G^z_t)_{iB} (G^z_t)_{aj},\label{eq_diff_1}\\
    &\frac{\partial (G^z_t)_{ij}}{\partial h^{(2)}_{aB}}=z(G^z_t)_{ia} (G^z_t)_{Bj}, \qquad \frac{\partial (G^z_t)_{ij}}{\partial \overline{h^{(2)}_{aB}}}=\overline{z} (G^z_t)_{iB} (G^z_t)_{aj},\label{eq_diff_2}
    \end{align}
for any $i,j \in \llbracket 2n \rrbracket $, $a \in 
\llbracket n \rrbracket$, and 
$B\in \llbracket n+1,2n \rrbracket$.

Since $F$ is continuously differentiable up to fifth order, by Itô's formula again and taking the expectation, we have
\begin{align}\label{eq_derivative}
\dd \mathbb{E}\big[ F(\II_t)\big]=\mathbb{E} \Big[ F'(\II_t) \big( \Theta_1+\Theta_2+\Theta_3 \big) \Big] \dd t + \frac{1}{2} \mathbb{E}\Big[ F''(\II_t) \big\< \dd M_t, \dd M_t  \big\>\Big],
\end{align}
using the martingale property of Brownian motions. 
From
\eqref{eq_diffusion} and the 
covariations 
$\langle \mathrm d \mathcal M_{aB}^{(\ell)}, 
\mathrm d \overline{\mathcal M}_{a'B'}^{(\ell')} \rangle
= \delta_{\ell,\ell'} \delta_{a,a'} \delta_{B,B'}$ 
and
$\langle \mathrm d \mathcal M_{aB}^{(\ell)}, 
\mathrm d \mathcal M_{a'B'}^{(\ell')} \rangle
= 0$ with $a \in 
\llbracket n \rrbracket$ and 
$B\in \llbracket n+1,2n \rrbracket$, we obtain
\begin{align}\label{eq_variation}
\frac{1}{2} \mathbb{E}\Big[ F''(\II_t) \big\< \dd M_t, \dd M_t  \big\>\Big]=&
\frac{1}{n} \sum_{l=1}^2 \mathbb{E} \left[F''(\II_t)
 \sum_{a=1}^n \sum_{B=n+1}^{2n}  \left(   \frac{ \partial \II_t}{\partial h^{(l)}_{aB}}      \frac{\partial \II_t}{\partial \overline{h^{(l)}_{aB}}} \right)\right] \times \dd t.
\end{align}
To compute the first drift term in \eqref{eq_derivative}, we apply the cumulant expansion formula (for a formal statement, see \cite[Lemma 7.1]{zbMATH06775355}) with respect to $\big(h^{(l)}_{aB}\big)$. Recall that the $(p,q)$-cumulant of a complex-valued random variable $\chi$ with finite moments are defined to be 
\begin{align}\label{cumulant_pq}
	\kappa^{(p,q)}(\chi)=(-\ii)^{p+q} \Big( \frac{\partial^{p+q}}{\partial s^p \partial t^q} \log \mathbb{E} \big[\mathrm{e}^{\ii s \chi+\ii t \overline{\chi}} \big] \Big) \Big|_{s,t=0},
\end{align}
provided the right side above exists. From \ref{it:condmom4}-\ref{it:condmom}, we find that the second cumulants of $\big(h^{(l)}_{aB}\big)$ are invariant over time, \ie
\begin{align}\label{eq_var}
	\kappa^{(1,1)}\big(h^{(l)}_{aB}\big)=\frac{1}{n}, \qquad \kappa^{(0,2)}\big(h^{(l)}_{aB}\big)=\kappa^{(2,0)}\big(h^{(l)}_{aB}\big)=0, 
\end{align}
and the higher-order cumulants are bounded by
\begin{align}\label{eq_cumu_bound}
	|\kappa^{(p,q)}\big(h^{(l)}_{aB}\big)| \lesssim  n^{-2-c_0},\qquad p+q=3,4, \qquad |\kappa^{(p,q)}\big(h^{(l)}_{aB}\big)| \lesssim  n^{-\frac{p+q}{2}}, \qquad p+q\geq 5,
\end{align}
uniformly for any $t\geq 0$. 
We then write out the second order terms in the expansions, \ie
\begin{align}
\mathbb{E} \big[ F'(\II_t) \Theta_1 \big]=&-\frac{1}{2}\sum_{a=1}^n \sum_{B=n+1}^{2n} \sum_{l=1}^2  \mathbb{E} \left[ h^{(l)}_{aB}  F'(\II_t) \frac{  \partial \II_t}{\partial h^{(l)}_{aB}}  \right]\nonumber\\
=&-\frac{1}{2n}\sum_{a=1}^n \sum_{B=n+1}^{2n} \sum_{l=1}^2  \mathbb{E} \left[ \frac{\partial}{\partial \overline{h^{(l)}_{aB}}} \left\{ F'(\II_t) \frac{  \partial \II_t}{\partial h^{(l)}_{aB}} \right\}  \right]
+\mbox{higher order terms}\label{second_order},
\end{align}
 using \eqref{eq_var}. The second drift term in \eqref{eq_derivative} with $\Theta_2$ can be handled very similarly with $h^{(l)}_{aB}$ replaced by $\overline{h^{(l)}_{aB}}$. Then we observe that the above second order cumulant expansion terms will cancel precisely with the third drift term in \eqref{eq_drift} and the quadratic variation term  in \eqref{eq_variation}.
This key cancellation is a consequence of the fact that the second moments of the matrix entries in \eqref{OU} do not change with time. Therefore, the remaining terms in \eqref{eq_derivative} are the higher order cumulant expansion terms in \eqref{second_order}. More precisely, the cumulant expansions starting with the third-order terms ($p+q+1=3$) are given by
\begin{align}\label{eq_cumulant_exp}
	\frac{\dd \mathbb{E}\big[ F(\II_t)\big]}{\dd t}
	=&-\frac{1}{2} \sum_{a=1}^n \sum_{B=n+1}^{2n}
	\sum_{l=1}^2 \left( \sum_{p+q+1= 3}^{4}\frac{\kappa^{(p+1,q)}(h^{(l)}_{aB})}{p!q!}
	\mathbb{E} \left[   \frac{\partial^{p+q+1} F(\II_t)}{\partial (h^{(l)}_{aB})^{p+1} \partial (\overline{h^{(l)}_{aB}})^{q} } \right] \right)\nonumber\\
	&-\frac{1}{2} \sum_{a=1}^n \sum_{B=n+1}^{2n} \sum_{l=1}^2 \left(\sum_{p+q+1 = 3}^{4}\frac{\kappa^{(p+1,q)}(\overline{h^{(l)}_{aB}})}{p!q!}
	\mathbb{E} \left[   \frac{\partial^{p+q+1} F(\II_t) }{\partial (\overline{h^{(l)}_{aB}})^{p+1} \partial (h^{(l)}_{aB})^{q}}\right]\right)+O(n^{-1/2+c(\epsilon+\nu)}),
\end{align}
 where $\kappa^{(p+1,q)}(\cdot)$ are the cumulants as bounded in \eqref{eq_cumu_bound} and we stopped the expansions at the fourth order with $p+q+1=4$ such that the truncation error is bounded by $O(n^{-1/2+c(\epsilon+\nu)})$ for some constant $c>0$. To estimate the truncation error, we used the cumulant bounds in \eqref{eq_cumu_bound}, the differentiation rules in \eqref{eq_diff_1}-\eqref{eq_diff_2}, the deterministic bound $\|G_t^z(\ii \eta)\| \leq \eta^{-1}$, together with the estimates in Lemma \ref{lemma_diff} below.  Such truncation argument is standard and frequently used in previous works \eg \cite{zbMATH06775355,zbMATH07549410,cipolloni2024universalityextremaleigenvalueslarge}, so we omit the details. We remark that the last error term holds uniformly for any $0\leq t\leq 100 \log n$. 

To estimate the third and higher order terms in \eqref{eq_cumulant_exp}, we introduce the following lemma, whose proof is postponed until the end of this section.
\begin{lemma}\label{lemma_diff}
Recall $\II_t$ in \eqref{def_L} with $\eta_0=n^{-1-\epsilon}$ and compactly supported $f_n$ satisfying $\int |\Delta_z f_n (z)| \dd^2 z \leq n^{\nu}$. Then
for any $k=k_1+k_2\geq 1$, we have
\begin{align}\label{eq_bound_k}
	\sup_{t\in[0,t_0]} \left|  \frac{\partial^k \II_t }{\partial (h^{(l)}_{aB})^{k_1} (\overline{h^{(l)}_{aB}})^{k_2}} \right| \prec n^{2k\epsilon+\nu}, \qquad l=1,2, \qquad t_0=100 \log n.
\end{align}
\end{lemma}

Using the cumulant bounds in \eqref{eq_cumu_bound} and the differentiation bounds in \eqref{eq_bound_k} to estimate the third and higher order terms on the right side of \eqref{eq_cumulant_exp}, we obtain that
\begin{align}\label{eq_derivative_F}
	\sup_{t\in[0,t_0]} \Big|\frac{\dd \mathbb{E}\big[ F(\II_t)\big]}{\dd t} \Big| =O(n^{-c_0/2}),
\end{align}
where we also used that $\epsilon<c_0/100$ and $\nu<10^{-4}c_0$. Integrating \eqref{eq_derivative_F} over $t \in [0, t_0]$ with $t_0=100\log n$ we obtain 
\begin{align}\label{eq_long_time}
	\Big| \mathbb{E}\big[ F(\II_0)\big]-\mathbb{E}\big[ F(\II_{t_0})\big] \Big| =O(n^{-c_0/4}). 
\end{align}
Note that $H^z_{t}$ is defined as in \eqref{xh} with the time dependent matrix $X^z_t=\A_t-z \B_t$, where $\A_t,\B_t$ are defined in \eqref{OU}. For any fixed $t\ge 0$, $X^z_t$ in distribution equals to
$$Y^z_t :{=} \mathrm{e}^{-\frac{t}{2}} X^z+\sqrt{1-\mathrm{e}^{-t}} X^{\mathrm{Gin}},$$
with $Y^z_0=X^z=(A-zB)/\sqrt{n}$ and $Y^z_\infty=(A^{\mathrm{Gin}}-zB^{\mathrm{Gin}})/\sqrt{n}$ being the Ginibre counterpart which is independent of $X^z$. We then use $G_t^{Y^z}$ to denote the resolvent of $Y_t^z$. Using the resolvent identity, we obtain
\begin{align}\label{eq_final}
	\|G^{Y^z}_{t_0}(\ii \eta)-G^{Y^z}_{\infty}(\ii \eta)\| \leq \|G^{Y^z}_{t_0}(\ii \eta)\| \|G^{Y^z}_{\infty}(\ii \eta)\| \|Y^z_{t_0}-Y^z_{\infty}\| 
	\prec n^{-50}, \qquad \eta \geq \eta_0,
\end{align}
where we also used that $\|G(\ii \eta)\| \leq \eta^{-1}$, $|x_{ij}| \prec n^{-1/2}$ from condition \ref{it:condmom} and $t_0=100\log n$. Thus using that $F$ has uniformly bounded derivatives up to fifth order and $\int|\Delta_z f(z)| \dd^2 z  \leq n^{\nu}$ with $\nu<10^{-4}c_0$,  we obtain from \eqref{def_L} that
\begin{align}\label{eq_long_time_2}
	\Big| \mathbb{E}\big[ F(\II_{t_0})\big]-\mathbb{E}^{\mathrm{Gin}}\big[ F(\II)\big] \Big| =O(n^{-10}). 
\end{align}
Combining \eqref{eq_long_time} with \eqref{eq_long_time_2}, we conclude
\begin{align}
	\Big| \mathbb{E}\big[ F(\II)\big]-\mathbb{E}^{\mathrm{Gin}}\big[ F(\II)\big] \Big|=\Big| \mathbb{E}\big[ F(\II_0)\big]-\mathbb{E}^{\mathrm{Gin}}\big[ F(\II)\big] \Big| =O(n^{-c_0/4}). 
\end{align}
We hence finish the proof of Lemma \ref{lemma_GFT}.	
\end{proof}

We end this section with the proof of Lemma \ref{lemma_diff}.

\begin{proof}[Proof of Lemma \ref{lemma_diff}]

	Recall $\II_t$ in \eqref{def_L} and the differentiation rules in \eqref{eq_diff_1}-\eqref{eq_diff_2}. Below we will present the proof for $l=1$ using \eqref{eq_diff_1}, and the proof for $l=2$ is very similar using \eqref{eq_diff_2} and that $f_n(z)=0$ for $|z|>C_0$. 
Using \eqref{eq_diff_1} and $G^2=-\ii \frac{\dd}{\dd \eta} G(\ii \eta)$, we have
\begin{align}\label{diff_L}
	\frac{\partial \II_t }{\partial h^{(l)}_{aB}}
	=&\frac{\ii}{4\pi} \int \Delta_z f_n(z) \int_{\eta_0}^{T}  \big[ \big(G_t^z(\ii \eta)\big)^2  \big]_{Ba}  \dd \eta \dd^2 z=-\frac{1}{4\pi} \int \Delta_z f_n(z)  \big( G^z_t(\ii \eta_0)\big)_{Ba} \dd^2 z+O_\prec(n^{-10}),
\end{align}
where we also used that $\|G(\ii \eta)\| \leq \eta^{-1}$ for $\eta=T=n^{100}$. Note that the variances of entries of $X^z_t$ do not change with time (see \eqref{eq_var}). Hence, for each fixed $t\geq 0$ and $|z|\leq C_0$, the local law as in \eqref{eq:local_law_G_z} also holds true for $G^z_t$ with very high probability.  To estimate the right side of \eqref{diff_L} with $\eta_0=n^{-1-\epsilon}$, we next extend the estimates in \eqref{eq:local_law_G_z} down to $\eta=n^{-1-\epsilon}$. 
Setting $\eta_1=n^{-1+\xi}$ with $\xi\leq \epsilon/100$ and using that $G^2=-\ii \frac{\dd}{\dd \eta} G(\ii \eta)$, we have
\begin{align}
	\Big|\big\langle \mathbf{x}, G^z(\ii \eta_1)\mathbf{y} \big\rangle-\big\langle \mathbf{x}, G^z(\ii \eta_0)\mathbf{y} \big\rangle\Big|=&\Big|\int_{\eta_0}^{\eta_1} \big\langle \mathbf{x}, \big(G^z(\ii \eta)\big)^2\mathbf{y} \big\rangle \dd \eta \Big|\nonumber\\
	\leq & \int_{\eta_0}^{\eta_1} \frac{1}{\eta}\sqrt{\big\langle \mathbf{x}, \Im G^z(\ii \eta) \mathbf{x} \big\rangle\big\langle \mathbf{y}, \Im G^z(\ii \eta) \mathbf{y} \big\rangle}   \dd \eta
\end{align}
for any deterministic unit vectors $\mathbf{x}, \mathbf{y}\in \mathbb{C}^{2n}$, where we also used the Cauchy-Schwarz inequality and the Ward identity $GG^*=\eta^{-1}\Im G$. By a simple spectral decomposition of $H^z$, it is easy to check that the function $\eta \big\langle \mathbf{x}, \Im G^z(\ii \eta) \mathbf{x} \big\rangle$ is increasing in $\eta>0$. Hence we obtain
\begin{align}
	\Big|\big\langle \mathbf{x}, G^z(\ii \eta_1)\mathbf{y} \big\rangle-\big\langle \mathbf{x}, G^z(\ii \eta_0)\mathbf{y} \big\rangle\Big|\leq & \sqrt{\big\langle \mathbf{x}, \Im G^z(\ii \eta_1) \mathbf{x} \big\rangle\big\langle \mathbf{y}, \Im G^z(\ii \eta_1) \mathbf{y} \big\rangle} \int_{\eta_0}^{\eta_1} \frac{\eta_1}{\eta^2} \dd \eta \prec  n^{2\epsilon}
\end{align}
where we also used that $\big\langle \mathbf{x}, \Im G^z(\ii \eta_1) \mathbf{x} \big\rangle \prec 1$ from \eqref{eq:local_law_G_z} and \eqref{rho} for $\eta_1=n^{-1+\xi}$ with $\xi<\epsilon/100$. Again using the local law from \eqref{eq:local_law_G_z} and \eqref{rho} for $\eta_1=n^{-1+\xi}$, $\mathbf{x}=\mathbf{e}_{x}$, and $\mathbf{y}=\mathbf{e}_{y}$, we conclude that, for each fixed $|z|\leq C_0$ and $t\geq 0$,
$$
\max_{x,y \in \llbracket 2n \rrbracket} \Big|\big( G^{z}_t(\ii \eta_0) \big)_{xy}\Big| \prec n^{2\epsilon}, \qquad \mbox{~with~} \quad \eta_0=n^{-1-\epsilon}.
$$
 Since $G^{z}_t$ is  H\"{o}lder continuous in both $z$ and $t$ almost surely from the resolvent identity and the flow in \eqref{OU}, one can show that the above bound holds true simultaneously for all $|z|\leq C_0$ and $t\in [0,t_0]$ with $t_0=100\log n$, \ie
\begin{align}\label{eq_G_bound}
	\sup_{t\in [0,t_0]} \sup_{|z|\leq C_0} \left\{ \max_{x,y \in \llbracket 2n \rrbracket} \Big|\big( G^{z}_t(\ii \eta_0) \big)_{xy}\Big| \right\} \prec n^{2\epsilon}, \qquad \mbox{~with~} \quad \eta_0=n^{-1-\epsilon},
\end{align}
 using a grid argument and taking the union bound (see similar explanations below \eqref{rho}).
Hence using that $\int|\Delta_z f_n(z)| \dd^2 z  \leq n^{\nu}$, we conclude from \eqref{diff_L} that
\begin{align}
	\sup_{t\in [0,t_0]} \Big|\frac{\partial \II_t }{\partial h^{(l)}_{aB}}\Big|  \prec n^{2\epsilon+\nu}.
\end{align} 
The same result also applies to $\partial /\partial \overline{h^{(l)}_{aB}}$. In general, for any $k\geq 1$, using the differentiation rules in \eqref{eq_diff_1}-\eqref{eq_diff_2} and the local law estimates in \eqref{eq_G_bound} repeatedly, we obtain the estimate for any $k$-th derivatives of $\II_t$ as stated in \eqref{eq_bound_k}. This ends the proof of Lemma \ref{lemma_diff}.
\end{proof}

\section{Lemmas on weak convergence}

We collect here classical results used in the proof of Theorem \ref{th:spherical:girko:gininf}
that we were not able to locate in a textbook, and we provide abridged proofs for the convenience of readers.

The first one states that we can consider compactly supported smooth functions instead of continuous ones in the definition of weak convergence. This is a natural and familiar variant of the Portmanteau theorem. The second one states that the convergence of point processes is characterized by compactly supported smooth test functions. The third and final one
is about the delta method in the context of point process.

\begin{lemma}[Characterization of convergence in law] \label{lem:WeakSmoothCompactly}
		If $X,X_1,X_2,\ldots$ are random variables such that 
		\begin{equation}
			\label{eq:limit}
			\lim_{n \to \infty} 
		\mathbb E[F(X_n)]
		= \mathbb E[F(X)]
		\end{equation}
		for every $\mathcal{C}^\infty$ compactly supported $F: \mathbb R \to \mathbb R$,
    	then the sequence ${(X_n)}_{n\geq1}$ converges in law to $X$.
\end{lemma}
	
\begin{proof}
		First, notice that if \eqref{eq:limit} is true for	compactly supported smooth functions, then it is true for	compactly supported continuous functions.
		Indeed, any compactly supported continuous function
		$F: \mathbb R \to \mathbb R $  is a uniform limit of 
		compactly supported smooth functions
		$F_k$. Then, for every $k \geq 1$ we have
		\[
		\limsup_{n \to \infty} \mathbb E[F_k(X_n)] -
		\|F- F_k\|_{\infty}
		\leq 
		\liminf_{n \to \infty} \mathbb E[F(X_n)]
		\leq 
		\limsup_{n \to \infty} \mathbb E[F(X_n)]
		\leq \limsup_{n \to \infty} \mathbb E[F_k(X_n)] + 
		\|F- F_k\|_{\infty}\]
		which implies the desired equality by taking $k \to \infty$.
		
		Next, assuming that \eqref{eq:limit} is true for compactly supported smooth functions, 
		we may show that $(X_n)_{n \geq 1}$
		is tight.
		For this, we may take a smooth function 
		$g:\mathbb R \to [0,1]$
		supported on $[-1,1]$
		such that $g|_{[-1/2,1/2]} \equiv 1$.
		Using that $\lim_{M \to \infty} g(x/M) = 1$
		and that $g$ is bounded by $1$, we obtain that
		$\lim_{M \to \infty}\mathbb E[g(X/M)] = 1$.
		Then, for any $\varepsilon > 0$, we may find
		an $M>0$ such that  
		$\mathbb E[g(X/M)] > 1 - \varepsilon$.
		Finally, since
		$\lim_{n \to \infty} 
		\mathbb E[g(X_n/M)] = \mathbb E[g(X/M)]$, we know that 
		$\mathbb E[g(X_n/M)] > 1-\varepsilon$
		for $n$ large enough so that
		\[\mathbb P(|X_n| \leq M) \geq \mathbb E[g(X_n/M)] > 1-\varepsilon.\]
		Finally, if $F$ is a bounded continuous function,
		we may write $F(x) = F(x) g(x/M) + F(x)(1-g(x/M))$ so that
				\[
				\mathbb E[F(X)g(X/M)] -
		\|F\|_{\infty} 
		\sup_{n \geq 1}\mathbb P(|X_n| \geq M/2)
		\leq 
				\liminf_{n \to \infty} \mathbb E[F(X_n)]\]
		and 
		\[
		\limsup_{n \to \infty} \mathbb E[F(X_n)]
		\leq \mathbb E[F(X)g(X/M)] + 
				\|F\|_{\infty} 
		\sup_{n \geq 1}\mathbb P(|X_n| \geq M/2).\]
		Taking $M \to \infty$ we get the desired equality.
		
		\emph{We may have also done a quicker proof
		by using Prokhorov's theorem (since we already showed tightness)
		and noticing that a random variable is characterized
		by the expected values of smooth compactly supported functions}.
\end{proof}

\begin{lemma}[Point processes] \label{lem:ProcessesSmoothCompactly}
Let $\mathcal{X},\mathcal{X}_1,\mathcal{X}_2,\ldots$ be point processes on $\mathbb C$ such that
\[\int_{\mathbb C} f \mathrm d \mu_{\mathcal X_n} 
\xrightarrow[n \to \infty]{ \dd}
\int_{\mathbb C} f \mathrm d \mu_{\mathcal X}
\]
for every smooth compactly supported function $f: \mathbb C \to \mathbb R$.
Then, the same holds for every continuous compactly supported function $f:\mathbb{C}\to\mathbb{R}$,
in other words the sequence $(\mathcal X_n)_{n \geq 1}$ converges in law towards $\mathcal X$.
\end{lemma}

\begin{proof}
Let us show that for all bounded Lipschitz $F: \mathbb R \to \mathbb R$ and continous and compactly supported $f: \mathbb C \to \mathbb R$,
\[\lim_{n \to \infty}\mathbb E\Big[F\Big(\int_{\mathbb C} f \mathrm d \mu_{\mathcal X_n} 
\Big) \Big]=
\mathbb E\Big[F\Big(\int_{\mathbb C} f \mathrm d \mu_{\mathcal X} 
\Big) \Big].
\]
Let $(f_k)_{k \geq 1}$ 
be a sequence of smooth functions whose support 
is included in the common compact set $K$ with
$\|f-f_k\|_{\infty} \to 0$. 
By denoting
$\nu_n = \mu_{\mathcal X_n}$ and $\nu_{\infty} = \mu_{\mathcal X}$,
we can write
\[
\Big|\mathbb{E}\Bigr[F\Bigr(\int f\mathrm d \nu_n\Bigr)\Bigr]
-
\mathbb{E}\Bigr[F\Bigr(\int f \mathrm d \nu_\infty\Bigr)\Bigr] \Big|
\leq T_n^{(k)} + T_\infty^{(k)} + S_n^{(k)},
\]
where
\[T_n^{(k)} = \Big|\mathbb{E}\Bigr[F\Bigr(\int f
\mathrm d \nu_n\Bigr)\Bigr]
-
\mathbb{E}\Bigr[F\Bigr(\int f_k \mathrm d \nu_n\Bigr)\Bigr] \Big|
\quad \mbox{ and }\quad 
S^{(k)}_n = \Big|\mathbb{E}\Bigr[F\Bigr(\int f_k
\mathrm d \nu_n\Bigr)\Bigr]
-
\mathbb{E}\Bigr[F\Bigr(\int f_k \mathrm d \nu_\infty\Bigr)\Bigr] \Big|.
 \] 
Since $F$ is Lipschitz and bounded, there is a constant $C>0$ 
such that
$|F(x) - F(y)| \leq C \min \{|x-y|,1\} $. We have
\[\limsup_{n \to \infty} T_n^{(k)} \leq 
\limsup_{n \to \infty} \mathbb E\big[C\min\big\{\|f - f_k \|_{\infty}
\nu_n(K), 1\big\} \big]
\leq 
\limsup_{n \to \infty} 
\mathbb E \Big[H^{(k)}\Big(\int g \mathrm d \nu_n\Big) \Big],\]
where we have 
taken a compactly supported smooth function
$g: \mathbb C \to \mathbb R$ such that
$1_K \leq g$ and we have
defined $H^{(k)}(x) = C \min\{\|f-f_k\|_{\infty} |x|,1\}$.
Since $H^{(k)}$ is continuous and bounded we get
\[\limsup_{n \to \infty} T_n^{(k)} \leq 
\mathbb E \Big[H^{(k)}\Big(\int g \mathrm d \nu_\infty\Big) \Big]\]
so that we have
\[\limsup_{n \to \infty} \Big|\mathbb{E}\Bigr[F\Bigr(\int f\mathrm d \nu_n\Bigr)\Bigr]
-
\mathbb{E}\Bigr[F\Bigr(\int f \mathrm d \nu_\infty\Bigr)\Bigr] \Big|
\leq 
2\mathbb E \Big[H^{(k)}\Big(\int g \mathrm d \nu_\infty\Big) \Big].\]
Since $H^{(k)}$ converges pointwise to zero as $k$ goes to $\infty$, we may apply the dominated convergence theorem to
$H^{(k)}(\int g \mathrm d \nu_{\infty})$ to get
$\lim_{k \to \infty} 
2\mathbb E [H^{(k)}(\int g \mathrm d \nu_\infty) ] = 0$ so that
\[\limsup_{n \to \infty} \Big|\mathbb{E}\Bigr[F\Bigr(\int f\mathrm d \nu_n\Bigr)\Bigr]
-
\mathbb{E}\Bigr[F\Bigr(\int f \mathrm d \nu_\infty\Bigr)\Bigr] \Big|
= 0,\]
which concludes the proof.

\end{proof}

The following lemma is a consequence of the
classical fact that if $f_n:M \to M$
is a sequence of measurable functions on
a locally compact metric space $M$ converging uniformly on
compact sets to the identity function and
$\mathcal X_n$ is a sequence of point processes
converging to $\mathcal X$
then $f_n(\mathcal X_n)$ also converges to $\mathcal X$. 
For convenience
of the reader, we give
a direct proof in the particular case
that interests us which can be seen
as a ``delta method'' lemma. 

\begin{lemma}[Delta-method for point processes] \label{lem:ChangeOfCoordinates}
Let $(\mathcal X_n)_{n \geq 1}$ be a sequence
of point processes on an open set 
$U \subset \mathbb C$ and fix $x_0 \in U$. Consider two open sets $V_1,V_2 \subset \mathbb C$
containing the origin 
together with diffeomorphisms
$\varphi_i: U \to V_i$.
 Then
$\sqrt n (\varphi_1(\mathcal X_n) - 
\varphi_1(x_0))$
converges in law if and only if $\sqrt n (\varphi_2(\mathcal X_n) - \varphi_2(x_0))$
also does and, in that case,
\[\lim_{n \to \infty}
\sqrt n \big(\varphi_1(\mathcal X_n) - \varphi_1(x_0) \big)
=
(\mathrm d \varphi_1)_{x_0} \circ
(\mathrm d \varphi_2)_{x_0}^{-1}
\Big(\lim_{n \to \infty}
\sqrt n \big(\varphi_2(\mathcal X_n) - \varphi_2(x_0)\big)\Big).\]
\end{lemma}

\begin{proof}
First, by changing $\varphi_i$ by 
$\varphi_i - \varphi_i(x_0)$, we may consider $\varphi_i(x_0)=0$.
Then, under this assumption, by considering the point process 
$\mathcal Y_n = \varphi_1(\mathcal X_n)$ 
and the maps
$\psi_i = \varphi_i \circ \varphi_1^{-1}$
the statement reduces to the following statement.
Let $\mathcal Y_n$ be a sequence of point
processes on $V_1$ that contains the origin 
such that 
$\sqrt n \mathcal Y_n$ converges. If
$\psi: V_1 \to V_2$ is a diffeomorphism
that satisfies $\psi(0) = 0$ then
\[\lim_{n \to \infty} 
\sqrt n\psi(\mathcal Y_n)
=(\mathrm d \psi)_0\big(\lim_{n \to \infty} \sqrt n \mathcal Y_n \big) .\]
To show this, let us consider 
a compactly supported smooth function $f: \mathbb C \to \mathbb R$
and let us study the difference
$\sum_{y \in \mathcal Y_n}
f(\sqrt n \psi(y)) - 
\sum_{y \in \mathcal Y_n}
f(\sqrt n \mathrm d\psi_0(y))$.
Write $\psi(x) = \mathrm d \psi_0 (x) + 
\|x\| r(x)$, where
the limit of $r(x)$ as $x$ goes to zero is zero.
Since $f$ is smooth and compactly supported
it is Lipschitz with 
some Lipschitz constant $C > 0$ and its
support is contained in some disk $\mathcal D$.
Using that $\psi^{-1}$ is Lipschitz near 
the origin we can take a disk 
$\widetilde {\mathcal D}$
such that 
$\psi^{-1}(\mathcal D/\sqrt n) \subset 
\widetilde {\mathcal D}/\sqrt n $
and we may also assume
$\mathrm d \psi_0^{-1}(\mathcal D/\sqrt n) \subset 
\widetilde {\mathcal D}/\sqrt n $.
Consider a compactly supported continuous
function
$g: \mathbb C \to \mathbb R$ satisfying $1_{\widetilde {\mathcal D}} \leq g$. We may write
\begin{align}
\Big|\sum_{y \in \mathcal Y_n}
f(\sqrt n \psi(y)) - 
\sum_{y \in \mathcal Y_n}
f(\sqrt n \mathrm d\psi_0(y))\Big|
&\leq \sum_{y \in \mathcal Y_n}
C \sqrt n|\psi(y) - \mathrm d \psi_0(y)|  
1_{y \in \psi^{-1}(\mathcal D/\sqrt n) 
\cup \mathrm d \psi_0^{-1}(\mathcal D/\sqrt n)}
\nonumber \\
&\leq \sum_{y \in \mathcal Y_n}
C\sqrt n \|y\| |r(y)| 1_{\widetilde 
{\mathcal D}/ \sqrt n}(y)      
\label{eq:SumDelta}    .           
\end{align}
If $R$ is the radius of $\widetilde {\mathcal D}$,
we know that
$\sqrt n\|y\|  1_{\widetilde 
{\mathcal D}/ \sqrt n}(y)   \leq R$.
Since $\lim_{x \to 0} r(x) = 0$,
we know that $r(y) 1_{\widetilde 
{\mathcal D}/ \sqrt n}(y) \leq \varepsilon_n $
for a sequence $\varepsilon_n$ that goes to
zero. Then we can bound
\eqref{eq:SumDelta} by
\[CR\varepsilon_n\sum_{y \in \mathcal Y_n}
g(\sqrt n y) \xrightarrow[n \to \infty]{\dd} 0  .\]
This implies the result.
\end{proof}

\bibliographystyle{abbrvnat}
{\footnotesize\bibliography{radius}}

@article{Davies,
  author =       {Jain, Vishesh and Sah, Ashwin and Sawhney, Mehtaab},
  title =        {On the real {Davies}' conjecture},
  fjournal =     {The Annals of Probability},
  journal =      {Ann. Probab.},
  XISSN =          {0091-1798},
  volume =       49,
  number =       6,
  pages =        {3011--3031},
  year =         2021,
  language =     {English},
  XDOI =           {10.1214/21-AOP1522},
  zbMATH =       7467489,
  Zbl =          {1486.15013}
}

@article {MR1148410,
  AUTHOR =       {Kostlan, Eric},
  TITLE =        {On the spectra of {G}aussian matrices},
  NOTE =         {Directions in matrix theory (Auburn, AL, 1990)},
  JOURNAL =      {Linear Algebra Appl.},
  FJOURNAL =     {Linear Algebra and its Applications},
  VOLUME =       {162/164},
  YEAR =         1992,
  PAGES =        {385--388},
  XDOI =           {10.1016/0024-3795(92)90386-O}
}

@article {MR2018415,
  AUTHOR =       {Shirai, Tomoyuki and Takahashi, Yoichiro},
  TITLE =        {Random point fields associated with certain {F}redholm
                  determinants. {I}. {F}ermion, {P}oisson and boson point
                  processes},
  JOURNAL =      {J. Funct. Anal.},
  FJOURNAL =     {Journal of Functional Analysis},
  VOLUME =       205,
  YEAR =         2003,
  NUMBER =       2,
  PAGES =        {414--463},
  XDOI =           {10.1016/S0022-1236(03)00171-X}
}

@article {MR2489167,
  AUTHOR =       {Krishnapur, Manjunath},
  TITLE =        {From random matrices to random analytic functions},
  JOURNAL =      {Ann. Probab.},
  FJOURNAL =     {The Annals of Probability},
  VOLUME =       37,
  YEAR =         2009,
  NUMBER =       1,
  PAGES =        {314--346},
  XDOI =           {10.1214/08-AOP404}
}

@book {MR2552864,
  AUTHOR =       {Hough, J. Ben and Krishnapur, Manjunath and Peres, Yuval and
                  Vir\'ag, B\'alint},
  TITLE =        {Zeros of {G}aussian analytic functions and determinantal
                  point processes},
  SERIES =       {University Lecture Series},
  VOLUME =       51,
  PUBLISHER =    {American Mathematical Society, Providence, RI},
  YEAR =         2009,
  PAGES =        {x+154},
  XISBN =         {978-0-8218-4373-4},
  XDOI =           {10.1090/ulect/051}
}

@book {MR2641363,
  AUTHOR =       {Forrester, P. J.},
  TITLE =        {Log-gases and random matrices},
  SERIES =       {London Mathematical Society Monographs Series},
  VOLUME =       34,
  PUBLISHER =    {Princeton University Press, Princeton, NJ},
  YEAR =         2010,
  PAGES =        {xiv+791},
  XISBN =         {978-0-691-12829-0},
  XDOI =           {10.1515/9781400835416}
}

@article {MR2684367,
  AUTHOR =       {Tao, Terence and Vu, Van},
  TITLE =        {Smooth analysis of the condition number and the least
                  singular value},
  JOURNAL =      {Math. Comp.},
  FJOURNAL =     {Mathematics of Computation},
  VOLUME =       79,
  YEAR =         2010,
  NUMBER =       272,
  PAGES =        {2333--2352},
  XDOI =           {10.1090/S0025-5718-2010-02396-8}
}

@article {MR2772389,
  AUTHOR =       {Bordenave, Charles},
  TITLE =        {On the spectrum of sum and product of non-{H}ermitian random
                  matrices},
  JOURNAL =      {Electron. Commun. Probab.},
  FJOURNAL =     {Electronic Communications in Probability},
  VOLUME =       16,
  YEAR =         2011,
  PAGES =        {104--113},
  XDOI =           {10.1214/ECP.v16-1606}
}

@article {MR2908617,
  AUTHOR =       {Bordenave, Charles and Chafa\"i, Djalil},
  TITLE =        {Around the circular law},
  JOURNAL =      {Probab. Surv.},
  FJOURNAL =     {Probability Surveys},
  VOLUME =       9,
  YEAR =         2012,
  PAGES =        {1--89},
  XDOI =           {10.1214/11-PS183}
}

@article {MR2926763,
  AUTHOR =       {Hardy, Adrien},
  TITLE =        {A note on large deviations for 2{D} {C}oulomb gas with
                  weakly confining potential},
  JOURNAL =      {Electron. Commun. Probab.},
  FJOURNAL =     {Electronic Communications in Probability},
  VOLUME =       17,
  YEAR =         2012,
  PAGES =        {no. 19, 12},
  XDOI =           {10.1214/ECP.v17-1818}
}

@article {MR3215627,
  AUTHOR =       {Chafa\"i, Djalil and P\'ech\'e, Sandrine},
  TITLE =        {A note on the second order universality at the edge of
                  {C}oulomb gases on the plane},
  JOURNAL =      {J. Stat. Phys.},
  FJOURNAL =     {Journal of Statistical Physics},
  VOLUME =       156,
  YEAR =         2014,
  NUMBER =       2,
  PAGES =        {368--383},
  XDOI =           {10.1007/s10955-014-1007-x}
}

@article {MR3306005,
  AUTHOR =       {Tao, Terence and Vu, Van},
  TITLE =        {Random matrices: universality of local spectral statistics
                  of non-{H}ermitian matrices},
  JOURNAL =      {Ann. Probab.},
  FJOURNAL =     {The Annals of Probability},
  VOLUME =       43,
  YEAR =         2015,
  NUMBER =       2,
  PAGES =        {782--874},
  XDOI =           {10.1214/13-AOP876}
}

@article {MR3615091,
  AUTHOR =       {Jiang, Tiefeng and Qi, Yongcheng},
  TITLE =        {Spectral radii of large non-{H}ermitian random matrices},
  JOURNAL =      {J. Theoret. Probab.},
  FJOURNAL =     {Journal of Theoretical Probability},
  VOLUME =       30,
  YEAR =         2017,
  NUMBER =       1,
  PAGES =        {326--364},
  XDOI =           {10.1007/s10959-015-0634-8}
}

@article {MR4408512,
  AUTHOR =       {Bordenave, Charles and Chafa\"i, Djalil and Garc\'ia-Zelada,
                  David},
  TITLE =        {Convergence of the spectral radius of a random matrix
                  through its characteristic polynomial},
  JOURNAL =      {Probab. Theory Related Fields},
  FJOURNAL =     {Probability Theory and Related Fields},
  VOLUME =       182,
  YEAR =         2022,
  NUMBER =       {3-4},
  PAGES =        {1163--1181},
  XDOI =           {10.1007/s00440-021-01079-9}
}

@incollection {MR4680362,
  AUTHOR =       {Tikhomirov, Konstantin},
  TITLE =        {Quantitative invertibility of non-{H}ermitian random
                  matrices},
  BOOKTITLE =    {{I}nternational {C}ongress of {M}athematicians. {V}ol. 4.
                  {S}ections 5--8},
  PAGES =        {3292--3313},
  PUBLISHER =    {EMS Press, Berlin},
  YEAR =         2023,
  XISBN =         {978-3-98547-062-4; 978-3-98547-562-9; 978-3-98547-058-7}
}

@article {MR4761213,
  AUTHOR =       {Cipolloni, Giorgio and Erd{\H{o}}s, L{\'a}szl{\'o} and Xu,
                  Yuanyuan},
  TITLE =        {Precise asymptotics for the spectral radius of a large
                  random matrix},
  JOURNAL =      {J. Math. Phys.},
  FJOURNAL =     {Journal of Mathematical Physics},
  VOLUME =       65,
  YEAR =         2024,
  NUMBER =       6,
  PAGES =        {Paper No. 063302, 53},
  XDOI =           {10.1063/5.0209705}
}

@unpublished{cipolloni2024universalityextremaleigenvalueslarge,
  title =        {Universality of extremal eigenvalues of large random
                  matrices},
  author =       {Giorgio Cipolloni and László Erdős and Yuanyuan Xu},
  year =         2024,
  note =         {preprint
                  \href{https://arxiv.org/abs/2312.08325v3}{arXiv:2312.08325v3}}
}

@unpublished{cutfast,
  title =        {On the cutoff phenomenon for fast diffusion and porous
                  medium equations},
  author =       {Chafaï, Djalil and Fathi, Max and Simonov, Nikita},
  year =         2025,
  note =
                  {\href{https://arxiv.org/abs/2503.11770v1}{arXiv:2503.11770v1}}
}

@book{fubini,
  AUTHOR =       {Huybrechts, Daniel},
  TITLE =        {Complex geometry},
  SERIES =       {Universitext},
  NOTE =         {An introduction},
  PUBLISHER =    {Springer-Verlag, Berlin},
  YEAR =         2005,
  PAGES =        {xii+309},
  XISBN =         {3-540-21290-6},
  MRCLASS =      {32Qxx (14-01 32-01 32G05 32J25 53C55 53C56)},
  MRNUMBER =     2093043,
  MRREVIEWER =   {Richard\ P.\ Thomas},
}

@article{invert,
  author =       {Tikhomirov, Konstantin},
  title =        {Invertibility via distance for noncentered random matrices
                  with continuous distributions},
  fjournal =     {Random Structures \& Algorithms},
  journal =      {Random Struct. Algorithms},
  XISSN =          {1042-9832},
  volume =       57,
  number =       2,
  pages =        {526--562},
  year =         2020,
  language =     {English},
  XDOI =           {10.1002/rsa.20920},
  zbMATH =       7279076,
  Zbl =          {1456.60028}
}

@article{maltsev-osman,
  title =        {Bulk universality for complex non-Hermitian matrices with
                  independent and identically distributed entries},
  XDOI =           {10.1007/s00440-024-01321-0},
  journal =      {Probability Theory and Related Fields},
  publisher =    {Springer Science and Business Media LLC},
  author =       {Maltsev, Anna and Osman, Mohammed},
  year =         2024,
  month =        sep
}

@article{wegner,
  author =       {Erd{\H{o}}s, L{\'a}szl{\'o} and Ji, Hong Chang},
  title =        {Wegner estimate and upper bound on the eigenvalue condition
                  number of non-{Hermitian} random matrices},
  fjournal =     {Communications on Pure and Applied Mathematics},
  journal =      {Commun. Pure Appl. Math.},
  XISSN =          {0010-3640},
  volume =       77,
  number =       9,
  pages =        {3785--3840},
  year =         2024,
  language =     {English},
  XDOI =           {10.1002/cpa.22201},
  zbMATH =       7896932,
  Zbl =          {1543.60009}
}

@article{zbMATH02028588,
  author =       {Rider, B.},
  title =        {A limit theorem at the edge of a non-{Hermitian} random
                  matrix ensemble},
  fjournal =     {Journal of Physics A: Mathematical and General},
  journal =      {J. Phys. A, Math. Gen.},
  XISSN =          {0305-4470},
  volume =       36,
  number =       12,
  pages =        {3401--3409},
  year =         2003,
  language =     {English},
  XDOI =           {10.1088/0305-4470/36/12/331},
  zbMATH =       2028588,
  Zbl =          {1039.60037}
}

@book{zbMATH03176754,
  author =       {Gel'fand, I. M. and Minlos, R. A. and Shapiro, Z. Ya.},
  title =        {Representations of the rotation and Lorentz groups and their
                  applications. Translated from the Russian by G. Cummins and
                  T. Boddington.},
  language =     {English},
  publisher =    {Oxford- University Press},
  pages =        {366 pp xviii},
  year =         1963,
  zbMATH =       3176754,
  Zbl =          {0108.22005}
}

@article{zbMATH03864263,
  author =       {Cohen, Joel E. and Newman, Charles M.},
  title =        {The stability of large random matrices and their products},
  fjournal =     {The Annals of Probability},
  journal =      {Ann. Probab.},
  XISSN =          {0091-1798},
  volume =       12,
  pages =        {283--310},
  year =         1984,
  language =     {English},
  XDOI =           {10.1214/aop/1176993291},
  zbMATH =       3864263,
  Zbl =          {0543.60098}
}

@article{zbMATH03901742,
  author =       {Girko, V. L.},
  title =        {Circular law},
  fjournal =     {Teoriya Veroyatnoste{\u{\i}} i e{\"e} Primeneniya},
  journal =      {Teor. Veroyatn. Primen.},
  volume =       29,
  number =       4,
  pages =        {669--679},
  year =         1984
}

@article{zbMATH03940306,
  author =       {Bai, Z. D. and Yin, Y. Q.},
  title =        {Limiting behavior of the norm of products of random matrices
                  and two problems of {Geman}-{Hwang}},
  fjournal =     {Probability Theory and Related Fields},
  journal =      {Probab. Theory Relat. Fields},
  XISSN =          {0178-8051},
  volume =       73,
  pages =        {555--569},
  year =         1986,
  language =     {English},
  XDOI =           {10.1007/BF00324852},
  zbMATH =       3940306,
  Zbl =          {0586.60021}
}

@article{zbMATH03978046,
  author =       {Geman, Stuart},
  title =        {The spectral radius of large random matrices},
  fjournal =     {The Annals of Probability},
  journal =      {Ann. Probab.},
  XISSN =          {0091-1798},
  volume =       14,
  pages =        {1318--1328},
  year =         1986,
  language =     {English},
  XDOI =           {10.1214/aop/1176992372},
  zbMATH =       3978046,
  Zbl =          {0605.60037}
}

@article{zbMATH06125015,
  author =       {Forrester, Peter J. and Mays, Anthony},
  title =        {Pfaffian point process for the {Gaussian} real generalised
                  eigenvalue problem},
  fjournal =     {Probability Theory and Related Fields},
  journal =      {Probab. Theory Relat. Fields},
  volume =       154,
  number =       {1-2},
  pages =        {1--47},
  year =         2012,
  XDOI =           {10.1007/s00440-011-0361-8}
}

@article{zbMATH06221300,
  author =       {Knowles, Antti and Yin, Jun},
  title =        {The isotropic semicircle law and deformation of {Wigner}
                  matrices},
  fjournal =     {Communications on Pure and Applied Mathematics},
  journal =      {Commun. Pure Appl. Math.},
  XISSN =          {0010-3640},
  volume =       66,
  number =       11,
  pages =        {1663--1749},
  year =         2013,
  language =     {English},
  XDOI =           {10.1002/cpa.21450},
  zbMATH =       6221300,
  Zbl =          {1290.60004}
}

@article{zbMATH06261246,
  author =       {Cacciapuoti, Claudio and Maltsev, Anna and Schlein,
                  Benjamin},
  title =        {Local {Marchenko}-{Pastur} law at the hard edge of sample
                  covariance matrices},
  fjournal =     {Journal of Mathematical Physics},
  journal =      {J. Math. Phys.},
  XISSN =          {0022-2488},
  volume =       54,
  number =       4,
  pages =        {043302, 13},
  year =         2013,
  language =     {English},
  XDOI =           {10.1063/1.4801856},
  zbMATH =       6261246,
  Zbl =          {1282.15031}
}

@article{zbMATH06330939,
  author =       {Bourgade, Paul and Yau, Horng-Tzer and Yin, Jun},
  title =        {Local circular law for random matrices},
  fjournal =     {Probability Theory and Related Fields},
  journal =      {Probab. Theory Relat. Fields},
  XISSN =          {0178-8051},
  volume =       159,
  number =       {3-4},
  pages =        {545--595},
  year =         2014,
  language =     {English},
  XDOI =           {10.1007/s00440-013-0514-z},
  zbMATH =       6330939,
  Zbl =          {1301.15021}
}

@article{zbMATH06347297,
  author =       {Ajanki, Oskari Heikki and Erd{\H{o}}s, L{\'a}szlo and
                  Kr{\"u}ger, Torben},
  title =        {Local semicircle law with imprimitive variance matrix},
  fjournal =     {Electronic Communications in Probability},
  journal =      {Electron. Commun. Probab.},
  volume =       19,
  pages =        9,
  note =         {Id/No 33},
  year =         2014,
  language =     {English},
  XDOI =           {10.1214/ECP.v19-3121}
}

@article{zbMATH06425285,
  author =       {Rogers, Tim},
  title =        {Universal sum and product rules for random matrices},
  fjournal =     {Journal of Mathematical Physics},
  journal =      {J. Math. Phys.},
  volume =       51,
  number =       9,
  pages =        {093304, 15},
  year =         2010,
  XDOI =           {10.1063/1.3481569}
}

@article{zbMATH06775355,
  author =       {He, Yukun and Knowles, Antti},
  title =        {Mesoscopic eigenvalue statistics of {Wigner} matrices},
  fjournal =     {The Annals of Applied Probability},
  journal =      {Ann. Appl. Probab.},
  volume =       27,
  number =       3,
  pages =        {1510--1550},
  year =         2017,
  XDOI =           {10.1214/16-AAP1237}
}

@book{zbMATH06780221,
  author =       {Erd{\H{o}}s, L{\'a}szl{\'o} and Yau, Horng-Tzer},
  title =        {A dynamical approach to random matrix theory},
  fseries =      {Courant Lecture Notes in Mathematics},
  series =       {Courant Lect. Notes Math.},
  XISSN =          {1529-9031},
  volume =       28,
  year =         2017,
  publisher =    {Providence, RI: American Mathematical Society (AMS); New
                  York, NY: Courant Institute of Mathematical Sciences},
  XDOI =           {10.1090/cln/028}
}

@article{zbMATH07329432,
  author =       {Cipolloni, Giorgio and Erd{\H{o}}s, L{\'a}szl{\'o} and
                  Schr{\"o}der, Dominik},
  title =        {Edge universality for non-{Hermitian} random matrices},
  fjournal =     {Probability Theory and Related Fields},
  journal =      {Probab. Theory Relat. Fields},
  XISSN =          {0178-8051},
  volume =       179,
  number =       {1-2},
  pages =        {1--28},
  year =         2021,
  language =     {English},
  XDOI =           {10.1007/s00440-020-01003-7},
  zbMATH =       7329432,
  Zbl =          {1461.60008}
}

@article{zbMATH07493826,
  author =       {Butez, Raphael and Garc{\'{\i}}a-Zelada, David},
  title =        {Extremal particles of two-dimensional {Coulomb} gases and
                  random polynomials on a positive background},
  fjournal =     {The Annals of Applied Probability},
  journal =      {Ann. Appl. Probab.},
  volume =       32,
  number =       1,
  pages =        {392--425},
  year =         2022,
  XDOI =           {10.1214/21-AAP1682}
}

@article{zbMATH07549410,
  author =       {Schnelli, Kevin and Xu, Yuanyuan},
  title =        {Convergence rate to the {Tracy}-{Widom} laws for the largest
                  eigenvalue of {Wigner} matrices},
  fjournal =     {Communications in Mathematical Physics},
  journal =      {Commun. Math. Phys.},
  volume =       393,
  number =       2,
  pages =        {839--907},
  year =         2022,
  XDOI =           {10.1007/s00220-022-04377-y}
}

@book{zbMATH07645442,
  author =       {Needham, Tristan},
  title =        {Visual complex analysis. 25th anniversary edition. {With} a
                  new foreword by {Roger} {Penrose}},
  XISBN =         {978-0-19-286891-6; 978-0-19-286892-3; 978-0-19-196494-7},
  year =         2023,
  publisher =    {Oxford: Oxford University Press},
  language =     {English},
  XDOI =           {10.1093/oso/9780192868916.001.0001},
  zbMATH =       7645442,
  Zbl =          {1508.30001}
}

@book{zbMATH07925435,
  author =       {Byun, Sung-Soo and Forrester, Peter J.},
  title =        {Progress on the study of the {Ginibre} ensembles},
  fseries =      {KIAS Springer Series in Mathematics},
  series =       {KIAS Springer Ser. Math.},
  XISSN =          {2731-5142},
  volume =       3,
  XISBN =         {978-981-975172-3; 978-981-975175-4; 978-981-975173-0},
  year =         2025,
  publisher =    {Singapore: Springer},
  language =     {English},
  XDOI =           {10.1007/978-981-97-5173-0},
  zbMATH =       7925435,
  Zbl =          {1561.15001}
}

\end{document}